\documentclass[12pt,reqno]{amsart}
\usepackage[margin=1in]{geometry}
\usepackage{amsmath,amssymb,amsthm,graphicx,amsxtra, setspace}
\usepackage[utf8]{inputenc}
\usepackage{mathrsfs}
\usepackage{hyperref}
\usepackage{xcolor}
\usepackage{upgreek}
\usepackage{mathtools}
\usepackage{xcolor}
\usepackage{cite}
\allowdisplaybreaks

\usepackage[pagewise]{lineno}

\newtheorem{theorem}{Theorem}[section]

\newtheorem{lemma}[theorem]{Lemma}

\newtheorem{remark}[theorem]{Remark}

\let\originalleft\left
\let\originalright\right
\renewcommand{\left}{\mathopen{}\mathclose\bgroup\originalleft}
\renewcommand{\right}{\aftergroup\egroup\originalright}

\renewcommand{\d}{\/\mathrm{d}\/}

\def\w{\textbf{W}^{\varepsilon}_{{\theta}^{\varepsilon}}}

\def\L{\mathbb{L}}

\def\C{\mathrm{C}}
\def\f{\boldsymbol{f}}

\def\X{\mathbb{X}}

\def\v{\boldsymbol{v}}

\def\w{\boldsymbol{w}}

\def\u{\mathrm{U}}

\def\u{\boldsymbol{u}}
\def\H{\mathbb{H}}

\newcommand{\R}{\mathbb{R}}

\renewcommand{\d}{\/\mathrm{d}\/}

\newcommand{\Addresses}{{
		\footnote{
			\noindent \textsuperscript{1,2,3}Department of Mathematics, Indian Institute of Technology Roorkee-IIT Roorkee,
			Haridwar Highway, Roorkee, Uttarakhand 247667, INDIA.\par\nopagebreak

				\noindent 	\textit{e-mail:} \texttt{Pardeep Kumar: pkumar3@ma.iitr.ac.in.}

			\noindent 	\textit{e-mail:} \texttt{Kush Kinra: kkinra@ma.iitr.ac.in.}

			\textit{e-mail:} \texttt{Manil T. Mohan: maniltmohan@ma.iitr.ac.in, maniltmohan@gmail.com.}
			
			\noindent \textsuperscript{*}Corresponding author.
			
			\textit{Key words:} Kelvin-Voigt fluids equation, inverse problem, memory kernel, integral over determination condition, contraction mapping principle. 	
			
			Mathematics Subject Classification (2020): Primary 35R30; Secondary 45K05, 45Q05, 76D03.

}}}

\begin{document}
	
	\title[On the three dimensional Kelvin-Voigt Fluids]{A local in time existence and uniqueness result of an inverse problem for the Kelvin-Voigt fluids
		\Addresses}
	
	\author[P. Kumar, K. Kinra and M. T. Mohan]
	{Pardeep Kumar\textsuperscript{1}, Kush Kinra\textsuperscript{2} and Manil T. Mohan\textsuperscript{3*}}

	\maketitle
	
	\begin{abstract}
		In this paper, we consider an inverse problem for three dimensional viscoelastic fluid flow equations, which arises from the motion of Kelvin-Voigt fluids in bounded domains (a hyperbolic type problem). This inverse problem aims to reconstruct the velocity  and kernel of the memory term simultaneously, from the measurement described as the integral over determination condition. By using the contraction mapping principle in an appropriate space, a local in time existence and uniqueness result for the inverse problem of Kelvin-Voigt fluids are obtained. Furthermore, using similar arguments, a global in time existence and uniqueness results for an inverse problem of Oseen type equations are also achieved.
	\end{abstract}

	\section{Introduction}  \label{sec1}\setcounter{equation}{0}
	First we shall discuss the fundamental facts that we come across while dealing with inverse problems. In the literature for the same, it can be noticed that the  global in time results are hard to comprehend for inverse problems. Not only claiming the uniqueness and stability results, but also establishing the existence of a solution is also a challenging problem.	The general strategy is to obtain global in time existence and uniqueness results	is to find a-priori estimates in appropriate function spaces for the unknowns of the problem with proper assumptions on the data.

	\subsection{The physical model and direct problem} 
	We, in this paper, shall be studying an inverse problem for Kelvin-Voigt fluid flow equations with memory term. The original direct problem model is the so-called viscoelastic fluid flow equations arising from the Kelvin-Voigt model for the non-Newtonian fluid flows, which can be described by the following system of integrodifferential equations (cf. \cite{Mohan,Oskolkov,OS})
	\begin{align}\label{1.1}
	\partial_{t}{\u}-\mu_{1} \partial_{t} \Delta \u-\mu_{0} \Delta \u +(\u\cdot \nabla)\u -\int_{0}^{t}k(t-s)\Delta \u(s) \d s+\nabla p=\f, \  \mbox{ in } \ \mathcal{O}  \times (0,T), 
	\end{align}
	and  the incompressibility condition
	\begin{align}\label{1.2}
	\nabla \cdot \u=0, \   \mbox{ in } \ \mathcal{O}  \times (0,T), 
	\end{align}
	where
	\begin{align*}
	\mu_{1}=\frac{2 \kappa_{2}}{\lambda} \  \mbox{and} \  \mu_{0}=\frac {2}{\lambda}\bigg(\kappa_{1}-\frac{\kappa_{2}}{\lambda}\bigg), 
	\end{align*}
	and all quantities are positive. The quantities $ \lambda$ and $\kappa_{j}, \ j \in \{1,2\}$, denote the relaxation time and retardation time, respectively (cf. \cite{Mohan} for more details). The system \eqref{1.1}-\eqref{1.2} is supplemented with the following initial and boundary conditions:
	\begin{align}
	\u&=\u_{0},  \   \mbox{ in } \ \mathcal{O} \times \{0\},\label{1.3}\\
	\u&=\mathbf{0},   \  \ \mbox{ on } \ \partial \mathcal{O} \times [0,T). \label{1.4}
	\end{align}
	Here, $\mathcal{O}$ is a bounded domain in $\mathbb{R}^{3}$ with smooth boundary $\partial \mathcal{O}$, $T>0$ is a constant, $\u\in\mathbb{R}^3$ denotes the velocity vector, $p\in\mathbb{R}$ is the  pressure of the fluid and $\f$ is the external forcing.  The integral term of \eqref{1.1} with kernel $k$ is the memory or hereditary term, which describes the viscoelastic property of the non-Newtonian fluids. One can consider these equations as 3D Navier-Stokes-Voight equations or Oskolkov’s equations (cf. \cite{VKK}) with a memory term $-\int_{0}^{t}k(t-s)\Delta \u(s) \d s$. The presence of the memory term and $-\mu_{1} \partial_{t} \Delta \u$ makes the system \eqref{1.1}-\eqref{1.4}, a hyperbolic type (cf. \cite{Mohan}). In the literature, some authors have referred the system \eqref{1.1}-\eqref{1.4} as Kelvin-Voigt fluid flow equations with ``fading memory" to differentiate it with Navier-Stokes-Voight equations (\cite{Mohan,Mohan1,Oskolkov}, etc).  
	
	When kernel $k$ is given (for example $k(t)=\gamma e^{-\delta t}$, where $\gamma=\frac{2}{\lambda}\left(\nu-\frac{\kappa_1}{\lambda}+\frac{\kappa_2}{\lambda^2}\right)>0$ and $\delta=\frac{1}{\lambda}>0$, $\nu$ is the viscosity co-efficient), the global solvability results  for the system \eqref{1.1}-\eqref{1.4} in bounded domains are available in the literature, see \cite{Mohan,Oskolkov,OS,OS1,ZT}, etc and the references therein. Recently,  the author in \cite{Mohan} established the existence and uniqueness of weak as well as strong solutions for the kernel $k(t)=\gamma e^{-\delta t}$, using a local monotonicity property of the linear and non-linear operators and a localized version of the Minty Browder technique in bounded and unbounded domains like Poincar\'e domains.  By using an $m$-accretive quantization of the linear and nonlinear operators, the author established  the existence and uniqueness of strong solutions  also.  The asymptotic behavior of Kelvin-Voigt fluid flow equations is studied in \cite{Mohan1} by establishing the existence of a finite Hausdorff and fractal dimensional global attractor. 
	
	\subsection{The inverse problem}	In practical phenomena, the kernel of memory term in \eqref{1.1}  represents the physical property of non-Newtonian fluids, which is difficult to determine in advance. By taking additional restriction on the velocity field $\u$, we can reconstruct $k$, provided the velocity is taken on a suitable subset of the domain $\mathcal{O}$. We assume that such an additional restriction on $\u$ can be represented in integral form, called integral over determination condition, by
	\begin{align}\label{1.5}
	\int_{\mathcal{O}}(I-\mu_{1} \Delta)\varphi(x) \cdot \u(x,t) \d x =r(t), \ \ t \in [0,T],
	\end{align}
	where $r(t)$ is the measurement data representing the average velocity on the domain $\mathcal{O}$ and $\varphi$ is a given function representing the type of device used to measure the velocity. Throughout this paper, external force $\f$ is assumed to be zero.	The precise formulation of our inverse problem is as follows:
	
	\begin{itemize}
		\item  Determine the velocity $\u$ and kernel of memory  $k$ satisfying the system \eqref{1.1}-\eqref{1.5} with initial pressure $p(x,0)=p_{0}$.
	\end{itemize}
	Due to the memory term, the inverse problem of reconstructing the kernel of memory term from \eqref{1.1}-\eqref{1.5} is a system of partial integrodifferential equations. This inverse problem has two types of nonlinearities, first, the convective term $(\u\cdot\nabla)\u$ and second, since $k$ is unknown, the convolution $-\int_{0}^{t}k(t-s)\Delta \u(s) \d s$ is also contains nonlinearity. 
	
	For some scalar parabolic equations with a memory term, the inverse problem of reconstructing the kernel of memory term has been well studied by Colombo et. al. using analytic semigroup theory (see \cite{Colombo1,Colombo2,CG,CG1,CGL,CGL1,CGV,CL,CL1,CL2}, etc). In particular, the results pertaining to both local and global in time existence and uniqueness of solution for an evolution equation with the nonlinearity having desirable growth condition  has been discussed  by Colombo and Guidetti (see \cite{CG}). The authors have  presented a new method, where the convolution term is rewritten as a sum of two linear terms in the unknown. By this way, they were able to successfully handle the nonlinearities of the convolution term and the source term in the proof of global uniqueness. \
	
	For the nonlinear Oskolkov’s system, the inverse problem with unknown source function is studied in \cite{VEF}. Carleman estimate and its application of proving the Lipschitz stability of an inverse problem for the Kelvin-Voigt model is described in \cite{OYMMY}.  Taking the similar path as in \cite{CG}, the hyperbolic type of inverse problems have been investigated by several authors to obtain global in time existence and uniqueness results (see \cite{Colombo, LR, WL}, etc). In \cite{Colombo}, the authors studied an inverse problem for the strongly damped wave equation with memory to obtain the similar results. A global in time existence and uniqueness result using a fixed point theorem for a one-dimensional integrodifferential hyperbolic system, which arises from a simplified model of thermoelasticity, is established in \cite{WL}. The author in  \cite{LR} studied the identification of two locally in time two (smooth)  convolution kernels in a fully hyperbolic phase-field system  coupling two hyperbolic integro-differential equations. Recently, the authors in \cite{JFJN}  proved the local solvability of an inverse problem to the Navier-Stokes equation with memory term (Oldroyd models) using fixed point arguments.\
	
	Similar to the inverse parabolic problem studied in \cite{CG}, the difficulties of our inverse problem also arise due to the memory term and the convective (nonlinear) term in the equation. 
	Existence and uniqueness of the inverse problem plays a significant role, especially when establishing the reconstruction scheme. The basic ideas to prove the results of this work have been adapted from \cite{CG,JFJN}. Using the contraction mapping principle, we obtain the existence and uniqueness of local in time solutions  for our inverse problem of the Kelvin-Voigt fluids. The main hindrance for establishing   the global existence and uniqueness is the presence of convective  term $(\u \cdot \nabla)\u$ in the Kelvin-Voigt fluid equation \eqref{1.1}.  In order to show the effectiveness of the method described in \cite{CG}, we consider a similar inverse problem for Oseen type equations corresponding to Kelvin-Voigt fluids with memory term. First, we obtain local in time existence and uniqueness result for this inverse problem. Then by linearizing the convolution term, that is, by splitting the convolution term into unknown velocity $\v$, and kernel of memory term $k$ and then calculating a-priori estimates for $\v$ and  $k$, we obtain the existence and uniqueness of global in time results for our inverse problem for the Oseen type equations with memory. 
	
	The plan of the paper is as follows: In the next section, we reformulate the inverse problem in a concrete setting by providing the compatibility and regularity conditions on the data and state  some useful technical results. An equivalent form for the inverse problem is formulated in section \ref{sec3}. Using this equivalent form and  contraction mapping principle, the main  result on the local in time existence and uniqueness of solutions is established in  section \ref{sec4} (Theorem \ref{thm1}).  Oseen type equations corresponding to Kelvin-Voigt fluids with memory term is considered in section \ref{sec5}. As in the case of nonlinear problem, firstly,  we  reformulate the inverse problem and then provide an equivalent form of this inverse problem. Using this equivalent problem, we prove the results on local and global in time existence and uniqueness of solutions (Theorems \ref{Thm1}, \ref{Thm2}, \ref{Thm3}).

	\section{The Concrete Version of the Inverse Problem} \label{sec2}\setcounter{equation}{0}
	In this section, we reformulate the inverse problem mentioned in the introduction into a concrete setting. Then we introduce the compatibility and regularity conditions on the data so that we can obtain the well-posedness of the inverse problem. We assume an additional condition on velocity field $\u(\cdot)$ in the form of an integral representation given in  \eqref{1.5}. This additional condition also helps us to determine the kernel of memory term. We discuss the necessary preliminary materials needed to obtain the local in time existence and uniqueness result of our inverse problem in three dimensional bounded domains. 
	
	\subsection{Function spaces}	For any integers $m,p$, let us denote by $\mathbb{W}^{m,p}(\mathcal{O}):=\mathrm{W}^{m,p}(\mathcal{O};\mathbb{R}^{3})$ and $\mathrm{W}^{m,p}(0,T)=\mathrm{W}^{m,p}(0,T;\mathbb{R})$ for the usual Sobolev spaces defined for spatial variable and time variable, respectively. For any Banach space $\mathbb{X}$, the space $ \mathrm{L}^{p}(0,T;\mathbb{X})$ consists all Lebesgue measurable functions $\u \colon [0,T] \to \mathbb{X}$ with $$\|\u\|_{\mathrm{L}^{p}(0,T;\mathbb{X})}:=\left(\int_{0}^{T}\|\u(t)\|_{\mathbb{X}}^{p}\d t\right)^{1/p} <\infty, \  \mbox{for}  \ \ 1\leq p < \infty,$$ and $$\|\u\|_{\mathrm{L}^{\infty}(0,T;\mathbb{X})}:=\operatorname*{ess\,sup}_{0 \leq t \leq T}\|\u(t)\|_{\X} <\infty.$$The Sobolev space $ \mathrm{W}^{m,p}(0,T;\mathbb{X})$ consists of all functions $\u \in \mathrm{L}^{p}(0,T;\mathbb{X})$ such that $\partial_{t}^{\beta}\u$ exists in the weak sense and belongs to $\mathrm{L}^{p}(0,T;\mathbb{X}),$ for all $0 \leq \beta \leq m$. Let us represent	$\mathbb{H}^{m}(\mathcal{O}):=\mathbb{W}^{m,2}(\mathcal{O}),$ $ \mathrm{H}^{m}(0,T):= \mathrm{W}^{m,2}(0,T),$ and $ \mathrm{H}^{m}(0,T;\mathbb{X}):= \mathrm{W}^{m,2}(0,T;\mathbb{X})$. We define $\mathrm{L}^2_0(\mathcal{O}):=\left\{p\in\mathrm{L}^2(\mathcal{O};\R):\int_{\mathcal{O}}p(x)\d x=0\right\}$ and $a \wedge b :=\min\{a,b\}$, where $a, b \in \mathbb{R}$.

	\subsection{The inverse problem}\label{IP1}
	For $T>0$, the inverse problem is to determine $\tau \in (0,T]$,
	\begin{align}\label{2.1}
	\u \in \mathrm{H}^{2}(0,\tau;\mathbb{H}_{0}^{1}(\mathcal{O}) \cap \mathbb{H}^{2}(\mathcal{O}))\ \mbox{ and }\ k \in \mathrm{L}^{2}(0,\tau),
	\end{align}
	such that $(\u,k)$ satisfies the system:
	\begin{subequations}
		\begin{align}
		\nonumber \partial_{t}{\u}-\mu_{1} \partial_{t} \Delta \u-\mu_{0} \Delta \u +(\u\cdot \nabla)\u& -\int_{0}^{t}k(t-s)\Delta \u(s) \d s \\ +\ \nabla p&=\mathbf{0},  \ \  \mbox{ in } \ \mathcal{O} \times (0,\tau),\label{2.2}  \\
		\nabla \cdot \u&=0,  \ \ \mbox{ in } \  \mathcal{O} \times (0,\tau),\label{2.3} \\
		\u&=\mathbf{0},  \ \ \mbox{ on }\  \partial \mathcal{O} \times [0,\tau),\label{2.4}  \\
		\u&=\u_{0}, \ \mbox{ in }  \ \mathcal{O} \times \{0\},\label{2.5} \\
		\int_{\mathcal{O}}(I-\mu_{1} \Delta)\varphi(x) \cdot \u(x,t) \d x&=r(t), \  \ \text { in } \  (0,\tau).\label{2.6}
		\end{align}
	\end{subequations}
	We solve the above inverse problem under the following assumptions on the data: 
	\begin{itemize}
		\item [(A1)] $\u_{0} \in \mathbb{H}_{0}^{1}(\mathcal{O}) \cap \mathbb{H}^{2}(\mathcal{O}), \ \nabla \cdot \u_{0}=0, \ \ \mbox{ in }   \ \mathcal{O}.$
		\item[(A2)] $\varphi \in \mathbb{H}_{0}^{1}(\mathcal{O}) \cap \mathbb{H}^{2}(\mathcal{O}), \  \nabla \cdot \varphi=0, \ \ \mbox{ in }  \ \mathcal{O}.$
		\item[(A3)] $\alpha^{-1}:= \int_{\mathcal{O}} \varphi \cdot \Delta \u_{0} \d x \ne 0.$
		\item[(A4)] $\v_{0}:=(I-\mu_{1} \Delta)^{-1}\big(\mu_{0} \Delta \u_{0}-(\u_{0} \cdot \nabla)\u_{0}- \nabla p_{0}\big)  \in \mathbb{H}_{0}^{1}(\mathcal{O}),$ $\nabla \cdot \v_{0}=0, 
		\	\ \mbox{ in } \ \mathcal{O}, \\ \ p_{0} \in \mathrm{H}^{1}(\mathcal{O}).$
		\item[(A5)] $r \in \mathrm{H}^{2}(0,T), $ with  \begin{align*} \int_{\mathcal{O}}(I-\mu_{1} \Delta)\varphi(x) \cdot \u_{0}(x) \d x&=r(0), \\ \int_{\mathcal{O}}\big(\mu_{0} \Delta \u_{0}(x)-(\u_{0}(x) \cdot \nabla)\u_{0}(x)\big) \cdot \varphi(x) \d x &=r'(0).\end{align*}
	\end{itemize}

	\subsection{Important inequalities and results} The following inequalities and results are used frequently in this paper.

	\begin{lemma}[Gagliardo-Nirenberg interpolation inequality, Theorem 1, \cite{Nirenberg}]\label{G-N2}
		Let $\mathcal{O} \subset \mathbb{R}^{n} $ and $\u \in \mathbb{W}^{m,p}(\mathcal{O}), p \geq 1$ and fix $1\leq p,q \leq \infty$ and a natural number $m$. Suppose also that a real number $\theta$ and a natural number $j$ are such that
		\begin{align}\label{G-Nn} 
		\theta=\left(\frac{j}{n}+\frac{1}{q}-\frac{1}{r}\right)\left(\frac{m}{n}-\frac{1}{p}+\frac{1}{q}\right)^{-1}
		\end{align}
		and $\frac{j}{m} \leq \theta \leq1$. Then for any  $\u \in  \mathbb{W}^{m,p}(\mathcal{O})$, we have
		\begin{align}
		\|\nabla^{j}\u\|_{\mathbb{L}^{r}} \leq C \big\| \u\|_{\mathbb{W}^{m,p}}^{\theta}\|\u\|_{\mathbb{L}^{q}}^{1-\theta},
		\end{align}
		where the constant $C$ depends upon the domain $\mathcal{O},m,n$.
	\end{lemma}
	Let us take $j=0, m=2$, $r=n=3$ and $p=q=2$ in \eqref{G-Nn} to get $ \theta=\frac{3}{4}$ and
	\begin{align*}
	\|\u\|_{\mathbb{L}^{3}} \leq C \| \u\|_{\mathbb{H}^{2}}^{1/4}\|\u\|_{\mathbb{L}^{2}}^{3/4}.
	\end{align*}
	Now if we consider $j=1, m=2, r=2, n=3$ and $p=q=2$ in \eqref{G-Nn}, then we find $ \theta=\frac{1}{2}$ and
	\begin{align*}
	\|\nabla\u\|_{\mathbb{L}^{2}} \leq C \|\u\|_{\mathbb{H}^{2}}^{1/2}\|\u\|_{\mathbb{L}^{2}}^{1/2}.
	\end{align*}
	\begin{lemma}[Theorem 4.4, \cite{CG}]\label{Lemma1}
		
		Let $\mathbb{X}$ be a Banach space, $p \in (1,\infty), \ \tau \in \mathbb{R}^+$, $ k \in \mathrm{L}^{p}(0,\tau)$, and $f \in \mathrm{L}^{p}(0,\tau;\mathbb{X})$. Then $k \ast f \in \mathrm{L}^{p}(0,\tau;\mathbb{X})$ and
		\begin{align*}  
		\|k \ast f\|_{\mathrm{L}^{p}(0,\tau;\mathbb{X})} \leq \tau^{1-1/p}\|k\|_{\mathrm{L}^{p}(0,\tau)}\|f\|_{\mathrm{L}^{p}(0,\tau;\mathbb{X})},
		\end{align*} 
		where $(k \ast f)(t):= \int_{0}^{t}k(t-s)f(s)\d s.$
	\end{lemma} 
	\begin{proof}
		Using Young's inequality for convolution and H\"older's inequality, we have
		\begin{align*}
		\|k \ast f\|_{\mathrm{L}^{p}(0,\tau;\mathbb{X})} &\leq \|k\|_{\mathrm{L}^{1}(0,\tau)}\|f\|_{\mathrm{L}^{p}(0,\tau;\mathbb{X})}\leq  
		\tau^{1-1/p}\|k\|_{\mathrm{L}^{p}(0,\tau)}\|f\|_{\mathrm{L}^{p}(0,\tau;\mathbb{X})},
		\end{align*}
		which completes the proof.
	\end{proof}
	\begin{lemma}[Theorem 4.5, \cite{CG}]\label{Lemma2}
		Let $\mathbb{X}$ be a Banach space, $p \in (1, \infty), \  \tau \in \mathbb{R}^+$, $z \in \mathrm{W}^{1,p}(0,\tau;\mathbb{X})$ with $z(0)=0$. Then 
		\begin{align*}
		\|z\|_{\mathrm{L}^{\infty}(0,\tau;\mathbb{X})} &\leq \tau^{1-1/p}\|\partial_{t}{z}\|_{\mathrm{L}^{p}(0,\tau;\mathbb{X})}, \\
		\|z\|_{\mathrm{L}^{p}(0,\tau;\mathbb{X})} &\leq \tau\|\partial_{t}{z}\|_{\mathrm{L}^{p}(0,\tau;\mathbb{X})}.
		\end{align*}
	\end{lemma}
	The proof can be easily concluded from H\"older's inequality and Young's inequality for convolution, using the formula $z=1\ast \partial_{t}z$.
	
	\section{The Equivalent Problem}\label{sec3}\setcounter{equation}{0}
	In this section, we transform the original inverse problem \eqref{2.2}-\eqref{2.6} into an equivalent system, using which we provide the proof of main result in the next section. Assuming the compatibility and regularity conditions (A1)-(A5), such a transformation is possible, and the following theorem gives an equivalent formulation.  
	\begin{theorem}\label{Equi-form}
		Let the assumptions $\emph{(A1)-(A5)}$ hold. Let $(\u,k)$ be a solution of the system \eqref{2.2}-\eqref{2.6} defined up to $T$ such that
		\begin{align}\label{3.1}
		\u \in \mathrm{H}^{2}(0,T;\mathbb{H}_{0}^{1}(\mathcal{O}) \cap \mathbb{H}^{2}(\mathcal{O})) \ \mbox{ and } \ k \in \mathrm{L}^{2}(0,T).
		\end{align}
		Then $\v:=\partial_{t}{\u}$ and $k$ verify the conditions
		\begin{align}\label{3.2}
		\v \in \mathrm{H}^{1}(0,T;\mathbb{H}_{0}^{1}(\mathcal{O}) \cap \mathbb{H}^{2}(\mathcal{O}))\ \mbox{ and }\  k \in \mathrm{L}^{2}(0,T),
		\end{align}
		and solve the system
		\begin{subequations}
			\begin{align}
			\nonumber \partial_{t}{\v}-\mu_{1} \partial_{t} \Delta \v-\mu_{0} \Delta \v- k\Delta \u_{0} -\int_{0}^{t}k(t-s)\Delta \v(s)\d s&+\big(\v\cdot \nabla\big)\left(\u_{0}+\int_{0}^{t}\v(s) \d s\right)\\ +\left(\left(\u_{0}+\int_{0}^{t}\v(s) \d s\right) \cdot \nabla\right)\v+\nabla \partial_{t}p&=\mathbf{0}, \ \  \mbox{ in } \ \mathcal{O} \times (0,T), \label{3.3}\\
			\nabla \cdot \v&=0,  \  \  \mbox{ in }\  \mathcal{O} \times (0,T),\label{3.4} \\
			\v&=\mathbf{0}, \ \  \mbox{ on }\ \partial \mathcal{O} \times [0,T),\label{3.5} \\
			\v&=\v_{0},  \  \mbox{ in } \ \mathcal{O} \times \{0\},\label{3.6}
			\end{align}
		\end{subequations}
		with 
		\begin{equation}\label{3.7}\tag{3.3e}
		\begin{aligned}
		& k(t)=\alpha \bigg\{r''(t)-\mu_{0} \int_{\mathcal{O}}\v \cdot \Delta \varphi \d x -\int_{0}^{t}\int_{\mathcal{O}}k(t-s) \v \cdot\Delta \varphi \d x \d s\\&-\int_{\mathcal{O}}\big((\v\cdot \nabla)\varphi\big) \cdot \left(\u_{0}
		+\int_{0}^{t}\v(s) \d s\right)\d x-\int_{\mathcal{O}}\left[\left(\left(\u_{0}+\int_{0}^{t}\v(s) \d s\right) \cdot \nabla\right) \varphi\right] \cdot \v \d x \bigg\}.
		\end{aligned}
		\end{equation}
		
		On the other hand, under the setting $\u(t):=\u_{0}+\int_{0}^{t}\v(s) \d s$, if $(\v,k)$ satisfies \eqref{3.2} and is a solution to the system \eqref{3.3}-\eqref{3.7}, then considering to the above setting, $(\u,k)$  satisfies \eqref{3.1} and is a solution to the system \eqref{2.2}-\eqref{2.6}. 
	\end{theorem}
	
	\begin{proof}
		The proof constitutes of mainly two steps.
		
		\vskip 0.2 cm
		\noindent\textbf{Step 1}. Suppose that the system \eqref{2.2}-\eqref{2.6} has a solution $(\u,k)$ satisfying
		$$\u \in \mathrm{H}^{2}(0,T;\mathbb{H}_{0}^{1}(\mathcal{O}) \cap \mathbb{H}^{2}(\mathcal{O}))\ \mbox{ and } \ k \in \mathrm{L}^{2}(0,T).$$
		Then, it is easy to see that $\v:=\partial_{t}{\u}$ and $k$ satisfy the equations \eqref{3.2}, \eqref{3.4} and \eqref{3.5}.
		From the equation \eqref{2.2}, we have
		\begin{align*}
		\partial_{t}{\u}&=(I-\mu_{1} \Delta)^{-1} \left(\mu_{0} \Delta \u -(\u\cdot \nabla)\u +\int_{0}^{t}k(t-s)\Delta \u(s)\d s-\nabla p\right),\\
		\v(x,0)&=(I-\mu_{1} \Delta)^{-1} \big(\mu_{0} \Delta \u_{0} -(\u_{0}\cdot \nabla)\u_{0}-\nabla p_{0}\big)=\v_{0}.
		\end{align*}
		Therefore, we obtain \eqref{3.6} using assumption (A4). Taking $\partial_{t}$ to equation \eqref{2.2}, we get
		\begin{align*}
		&\partial_{t}{\v}-\mu_{1} \partial_{t} \Delta \v-\mu_{0} \Delta \v +(\v\cdot \nabla)\u+(\u\cdot \nabla)\v -\partial_{t}\left(\int_{0}^{t}k(t-s)\Delta \u(s)\d s\right)+\nabla \partial_{t}p=\mathbf{0}.
		\end{align*}
		Using $\u(t)=\u_{0}+\int_{0}^{t}\v(s) \d s$ in the above equation, we find
		\begin{align*}
		&\partial_{t}{\v}-\mu_{1} \partial_{t} \Delta \v-\mu_{0} \Delta \v -k\Delta \u_{0} -\int_{0}^{t}k(t-s)\Delta \v(s)\d s\\&\quad+\big(\v\cdot \nabla\big)\left(\u_{0}+\int_{0}^{t}\v(s) \d s\right)+\left(\left(\u_{0}+\int_{0}^{t}\v(s) \d s\right)\cdot \nabla\right)\v+\nabla \partial_{t}p=\mathbf{0},
		\end{align*}
		and equation \eqref{3.3} follows. Taking the inner product in \eqref{3.3} with $\varphi$ and using assumptions (A2) and (A3), we obtain
		\begin{align*}
		k(t)&=\alpha \bigg\{ r''(t)-\mu_{0} \int_{\mathcal{O}} \v \cdot \Delta \varphi \d x -\int_{0}^{t}\int_{\mathcal{O}}k(t-s) \v \cdot \Delta \varphi \d x \d s\\&\quad-\int_{\mathcal{O}}\big((\v\cdot \nabla)\varphi\big)\cdot\left(\u_{0}+\int_{0}^{t}\v(s) \d s\right)\d x-\int_{\mathcal{O}}\left[\left(\left(\u_{0}+\int_{0}^{t}\v(s) \d s\right)\cdot \nabla\right)\varphi \right]\cdot\v \d x\bigg\},
		\end{align*}
		where $ r''(t)=\int_{\mathcal{O}}(I-\mu_{1} \Delta)\varphi \cdot \partial_{t} {\v} \d x$, and hence the equation \eqref{3.7} follows.
		
		\vskip 0.2 cm
		\noindent	\textbf{Step 2}. Suppose that the system \eqref{3.3}-\eqref{3.7} has a solution $(\v,k)$ satisfying
		$$\v \in \mathrm{H}^{1}(0,T;\mathbb{H}_{0}^{1}(\mathcal{O}) \cap \mathbb{H}^{2}(\mathcal{O})) \ \mbox{ and } \  k \in \mathrm{L}^{2}(0,T).$$	It can be easily shown that \eqref{3.1} and\eqref{2.3}-\eqref{2.5} hold.	Since $\v=\partial_{t} \u$, the equation \eqref{3.3} can be rewritten as
		\begin{align*}
		\partial_{t}\left(\partial_{t}{\u}-\mu_{1} \partial_{t} \Delta \u-\mu_{0} \Delta \u +(\u\cdot \nabla)\u -\int_{0}^{t}k(t-s)\Delta \u(s)
		\d s+\nabla p \right)=\mathbf{0}.
		\end{align*}
		Integrating the above equation, we obtain
		\begin{align*}
		\partial_{t}{\u}-\mu_{1} \partial_{t} \Delta \u-\mu_{0} \Delta \u +(\u\cdot \nabla)\u -\int_{0}^{t}k(t-s)\Delta \u(s)\d s+\nabla p=c_1,
		\end{align*}
		where $c_1$ is the constant of integration. Taking $t=0$, we find
		\begin{align*}
		\left(I-\mu_{1} \Delta\right) \v_{0}-\mu_{0} \Delta \u_{0} +(\u_{0}\cdot \nabla)\u_{0}+\nabla p_{0}=c_1.
		\end{align*}
		Making use of assumption (A4), we get $c_1=0$ and we obtain \eqref{2.2}. The equation \eqref{3.7} for $k$ can be re-written as
		\begin{align*}
		r''(t)&=\partial_{t}\left(\mu_{0} \int_{\mathcal{O}} \varphi \cdot \Delta \u \d x -\int_{\mathcal{O}}\big((\u\cdot \nabla)\u\big) \cdot \varphi \d x+\int_{0}^{t}\int_{\mathcal{O}}k(t-s) \varphi \cdot \Delta \u \d x \d s \right).
		\end{align*}
		Integrating this equation with respect to $t$, we obtain
		\begin{align*}
		r'(t)=\mu_{0} \int_{\mathcal{O}} \varphi \cdot \Delta \u \d x -\int_{\mathcal{O}}\big((\u\cdot \nabla)\u\big) \cdot \varphi \d x+\int_{0}^{t}\int_{\mathcal{O}}k(t-s) \varphi \cdot \Delta \u \d x \d s+c_2.
		\end{align*}
		Taking $t=0$ and using assumption (A5), we get $c_2=0,$ so that 
		\begin{align*}
		r'(t)=\mu_{0} \int_{\mathcal{O}} \varphi \cdot \Delta \u \d x -\int_{\mathcal{O}}\big((\u\cdot \nabla)\u\big) \cdot \varphi \d x+\int_{0}^{t}\int_{\mathcal{O}}k(t-s) \varphi \cdot \Delta \u \d x \d s.
		\end{align*}
		Using \eqref{2.2} in the above equation along with assumption (A2), we deduce that 
		\begin{align*}
		r'(t)&= \int_{\mathcal{O}} (I-\mu_{1}\Delta)\partial_{t} \u \cdot \varphi \d x+ \int_{\mathcal{O}} \nabla p \cdot \varphi \d x\\
		&= \int_{\mathcal{O}} (I-\mu_{1}\Delta)\varphi \cdot \partial_{t}\u \d x
		= \partial_{t}\left[\int_{\mathcal{O}} (I-\mu_{1}\Delta)\varphi \cdot \u \d x\right].
		\end{align*}
		Integrating the above equation, taking $t=0$ and using assumption (A5), we arrive at
		\begin{align*}
		r(t)= \int_{\mathcal{O}}  (I-\mu_{1}\Delta)\varphi \cdot \u \d x,
		\end{align*}
		which is \eqref{2.6} and it completes the proof.
	\end{proof}	

	\section{Local in Time Existence}\label{sec4}\setcounter{equation}{0}
	From Theorem \ref{Equi-form}, we know that  solving the inverse problem \eqref{2.2}-\eqref{2.6} with solution  $(\u,k)$ is equivalent to solving the system \eqref{3.3}-\eqref{3.7} with solution $(\v,k)$. Therefore, instead of proving results of the existence and uniqueness of solutions to the system \eqref{2.2}-\eqref{2.6}, we prove similar results for the system  \eqref{3.3}-\eqref{3.7}. The proof is based on \emph{contraction mapping principle}. It is important to note that having existence of unique $(\v,k)$ to the system \eqref{3.3}-\eqref{3.7} is equivalent to having existence of unique $(\u,k)$ to the inverse problem \eqref{2.2}-\eqref{2.6}.\\
	
	The following theorem is the main result of this section: 
	\begin{theorem}[Local in time existence]\label{thm1}
		Let the assumptions $(A1)$-$(A5)$ hold. Then there exists $\tau \in (0,T]$, such that the inverse problem \ref{IP1} has a unique solution $$(\u,k) \in \mathrm{H}^{2}(0,\tau;\mathbb{H}_{0}^{1}(\mathcal{O}) \cap \mathbb{H}^{2}(\mathcal{O})) \times \mathrm{L}^{2}(0,\tau).$$
	\end{theorem}
	\begin{proof}
		
		Let us define the space
		\begin{align*}
		\mathcal{V}(\tau,L):=&\bigg\{(\tilde{\v},\tilde{k}) \in  \mathrm{H}^{1}(0,\tau;\mathbb{H}_{0}^{1}(\mathcal{O}) \cap \mathbb{H}^{2}(\mathcal{O})) \times \mathrm{L}^{2}(0,\tau):
		\nabla \cdot \tilde{\v}=0, \ \mbox{ in } \ \mathcal{O} \times (0,\tau),\\
		&\quad\tilde{\v}=\mathbf{0}, \  \mbox{ on }\ \partial \mathcal{O} \times [0,\tau),\ \ 
		\tilde{\v}=\v_{0}, \    \mbox{ in } \ \mathcal{O} \times \{0\}  \ \mbox{ and } \ \nonumber\\&\quad \|\tilde{\v}\|_{\mathrm{H}^{1}\left(0,\tau;\mathbb{H}_{0}^{1}(\mathcal{O}) \cap \mathbb{H}^{2}(\mathcal{O})\right)} +\|\tilde{k}\|_{\mathrm{L}^{2}(0,\tau)} \leq L  \bigg\},
		\end{align*}
		where $L$ is a positive constant, which will be determined later. We also define the mapping $ \Psi \colon \mathcal{V}(\tau,L) \to \mathcal{V}(\tau,L)$ such that
		$(\tilde{\v},\tilde{k}) \mapsto (\v,k)$ 
		through \\
		\begin{align}\label{4.1}
		&k(t)\nonumber:=\alpha \bigg\{r''-\mu_{0} \int_{\mathcal{O}}\tilde{\v} \cdot \Delta \varphi \d x -\int_{0}^{t}\int_{\mathcal{O}}\tilde{k}(t-s) \tilde{\v} \cdot\Delta \varphi \d x \d s\\ &\quad-\int_{\mathcal{O}}\big((\tilde{\v}\cdot \nabla)\varphi\big) \cdot \bigg(\u_{0}
		+\int_{0}^{t}\tilde{\v}(s)\d s\bigg)\d x-\int_{\mathcal{O}}\left[\left(\left(\u_{0}+\int_{0}^{t}\tilde{\v}(s)\d s\right) \cdot \nabla\right) \varphi\right] \cdot \tilde{\v}\d x\bigg\},
		\end{align}
		and the initial boundary value problem
		\begin{equation}\label{4.2}
		\left\{
		\begin{aligned}
		\partial_{t}{\v}-\mu_{1} \partial_{t} \Delta \v-\mu_{0} \Delta \v- k\Delta \u_{0}-\int_{0}^{t}k(t-s)\Delta \tilde{\v}(s)\d s&+(\tilde{\v}\cdot \nabla)\bigg(\u_{0}+\int_{0}^{t}\tilde{\v}(s)\d s\bigg)\\+\left(\left(\u_{0}+\int_{0}^{t}\tilde{\v}(s)\d s\right) \cdot \nabla\right)\tilde{\v}+\nabla \partial_{t}p&=\mathbf{0}, \ \  \mbox{ in }\ \mathcal{O} \times (0,\tau), \\
		\nabla \cdot \v&=0, \ \   \mbox{ in }\   \mathcal{O} \times (0,\tau), \\
		\v&=\mathbf{0}, \ \  \mbox{ on }\ \partial \mathcal{O} \times [0,\tau), \\
		\v&=\v_{0}, \   \mbox{ in }\   \mathcal{O} \times \{0\}.
		\end{aligned}
		\right.
		\end{equation}

		In order to complete the proof of Theorem \ref{thm1}, we need to show that the mapping  $ \Psi \colon \mathcal{V}(\tau,L) \to \mathcal{V}(\tau,L)$ is a contraction map. For simplicity, hereafter $C$ represents a generic constant depending only on $\mathcal{O}$.
		\vskip 3mm
		\noindent 
		\textbf{Step 1}. Firstly, we show that the map $\Psi$ is well defined for an appropriate choice of $L$ and $\tau$. From \eqref{4.1}, we have 
		\begin{align}\label{4.1N}
		\left\|k\right\|_{\mathrm{L}^{2}(0,\tau)} &\leq \alpha \|r''\|_{\mathrm{L}^{2}(0,\tau)} +\alpha \mu_{0}\|\tilde{\v}\|_{\mathrm{L}^{2}(0,\tau;\mathbb{L}^{2}(\mathcal{O}))} \|\Delta \varphi\|_{\mathbb{L}^{2}(\mathcal{O})} \nonumber \\&\quad+\alpha\left\|\int_{0}^{t}\tilde{k}(t-s) \left(\int_{\mathcal{O}}\tilde{\v} \cdot\Delta \varphi \d x\right)\d s\right\|_{\mathrm{L}^{2}(0,\tau)}  \nonumber \\&\quad+2\alpha\left(\int_{0}^{\tau}\left\|\u_{0}+\int_{0}^{t}\tilde{\v}(s)\d s\right\|^2_{\mathbb{L}^{\infty}(\mathcal{O})} \left\|\tilde{\v}(t)\right\|^2_{\mathbb{L}^{2}(\mathcal{O})} \left\|\nabla \varphi\right\|^2_{\mathbb{L}^{2}(\mathcal{O})}\d t\right)^{1/2}.
		\end{align}
		Applying Lemma \ref{Lemma1}, we find
		\begin{align*}
		\left\|\int_{0}^{t}\tilde{k}(t-s) \left(\int_{\mathcal{O}}\tilde{\v} \cdot\Delta \varphi \d x\right)\d s\right\|_{\mathrm{L}^{2}(0,\tau)} 
		& \leq
		\left\|\tilde{k}\right\|_{\mathrm{L}^{1}(0,\tau)}\left\|\int_{\mathcal{O}}\tilde{\v} \cdot \Delta \varphi \d x\right\|_{\mathrm{L}^{2}(0,\tau)} \\
		& \leq \tau^{1/2}\|\tilde{k}\|_{\mathrm{L}^{2}(0,\tau)} \left( \int_{0}^{\tau}\|\tilde{\v}(t)\|^2_{\mathbb{L}^{2}(\mathcal{O})} \|\Delta \varphi\|^2_{\mathbb{L}^{2}(\mathcal{O})}\d t\right)^{1/2}\\
		&= \tau^{1/2}\|\tilde{k}\|_{\mathrm{L}^{2}(0,\tau)}\|\Delta \varphi\|_{\mathbb{L}^{2}(\mathcal{O})}\left\|\tilde{\v}\right\|_{\mathrm{L}^{2}\left(0,\tau;\mathbb{L}^{2}(\mathcal{O})\right)}.
		\end{align*}
		Substituting the above estimate in \eqref{4.1N}, we obtain
		\begin{align}\label{4.3}
		\nonumber \|k\|_{\mathrm{L}^{2}(0,\tau)} 
		& \leq \alpha \|r''\|_{\mathrm{L}^{2}(0,\tau)}+\alpha \bigg\{\mu_{0}\|\Delta \varphi\|_{\mathbb{L}^{2}(\mathcal{O})}+\tau^{1/2}\|\tilde{k}\|_{\mathrm{L}^{2}(0,\tau)}\left\|\Delta \varphi\right\|_{\mathbb{L}^{2}(\mathcal{O})}\\&\quad+2\left(\|\u_{0}\|_{\mathbb{L}^{\infty}(\mathcal{O})}+\int_{0}^{\tau}\|\tilde{\v}(s)\|_{\mathbb{L}^{\infty}(\mathcal{O})}\d s \right) \|\nabla \varphi\|_{\mathbb{L}^{2}(\mathcal{O})}\bigg\} \|\tilde{\v}\|_{\mathrm{L}^{2}\left(0,\tau;\mathbb{L}^{2}(\mathcal{O})\right)}.
		\end{align}
		Using the Sobolev embedding theorem, we deduce that 
		\begin{align}\label{4.4}
		\|\u_{0}\|_{\mathbb{L}^{\infty}(\mathcal{O})}+\int_{0}^{\tau}\|\tilde{\v}(s)\|_{\mathbb{L}^{\infty}(\mathcal{O})}\d s &\leq C \left(\|\u_{0}\|_{\mathbb{H}^{2}(\mathcal{O})}+\int_{0}^{\tau}\|\tilde{\v}(s)\|_{\mathbb{H}^{2}(\mathcal{O})}\d s\right) \nonumber\\
		&\leq  C \left(\|\u_{0}\|_{\mathbb{H}^{2}(\mathcal{O})}+\tau^{1/2}\|\tilde{\v}\|_{\mathrm{L}^{2}\left(0,\tau;\mathbb{H}^{2}(\mathcal{O})\right)}\right).
		\end{align}
		Applying Lemma \ref{Lemma2}, we get
		\begin{align}\label{4.5}
		\|\tilde{\v}\|_{\mathrm{L}^{2}(0,\tau;\mathbb{L}^{2}(\mathcal{O}))} &\leq \|\tilde{\v}-\v_{0}\|_{\mathrm{L}^{2}(0,\tau;\mathbb{L}^{2}(\mathcal{O}))}+\|\v_{0}\|_{\mathrm{L}^{2}(0,\tau;\mathbb{L}^{2}(\mathcal{O}))}\nonumber\\
		&\leq
		\tau\|\partial_{t}(\tilde{\v}-\v_{0})\|_{\mathrm{L}^{2}(0,\tau;\mathbb{L}^{2}(\mathcal{O}))}+\tau^{1/2}\|\v_{0}\|_{\mathbb{L}^{2}(\mathcal{O})}\nonumber\\
		&\leq
		\tau\|\partial_{t}\tilde{\v}\|_{\mathrm{L}^{2}(0,\tau;\mathbb{L}^{2}(\mathcal{O}))}+\tau^{1/2}\|\v_{0}\|_{\mathbb{L}^{2}(\mathcal{O})}.
		\end{align}
		Using \eqref{4.4} and \eqref{4.5} in \eqref{4.3} and then by the definition of the space $\mathcal{V}(\tau,L)$, we arrive at
		\begin{align}\label{4.6}
		\|k\|_{\mathrm{L}^{2}(0,\tau)}&\leq \alpha \|r''\|_{\mathrm{L}^{2}(0,\tau)}+\alpha\bigg\{ \mu_{0}\|\Delta \varphi\|_{\mathbb{L}^{2}(\mathcal{O})}+ \tau^{1/2} \|\tilde{k}\|_{\mathrm{L}^{2}(0,\tau)}\left\|\Delta \varphi\right\|_{\mathbb{L}^{2}(\mathcal{O})}\nonumber\\&\quad+C\bigg(\|\u_{0}\|_{\mathbb{H}^{2}(\mathcal{O})}+\tau^{1/2}\|\tilde{\v}\|_{\mathrm{L}^{2}(0,\tau;\mathbb{H}^{2}(\mathcal{O}))}\bigg) \|\nabla \varphi\|_{\mathbb{L}^{2}(\mathcal{O})}\bigg\} \nonumber\\&\qquad\times\left(\tau\|\partial_{t}\tilde{\v}\|_{\mathrm{L}^{2}(0,\tau;\mathbb{L}^{2}(\mathcal{O}))}+\tau^{1/2}\| \v_{0}\|_{\mathbb{L}^{2}(\mathcal{O})}\right)\nonumber\\
		&\leq \alpha \|r''\|_{\mathrm{L}^{2}(0,\tau)}+\alpha\bigg\{ \mu_{0}\|\Delta \varphi\|_{\mathbb{L}^{2}(\mathcal{O})}+ \tau^{1/2}L\|\Delta \varphi\|_{\mathbb{L}^{2}(\mathcal{O})}\nonumber \\&\quad+C\left(\|\u_{0}\|_{\mathbb{H}^{2}(\mathcal{O})}+\tau^{1/2}L\right)\|\nabla \varphi\|_{\mathbb{L}^{2}(\mathcal{O})}\bigg\}\tau^{1/2} \left(\tau^{1/2}L+\| \v_{0}\|_{\mathbb{L}^{2}(\mathcal{O})}\right).
		\end{align}
		Using the energy estimates for the linear problem \eqref{4.2}, we estimate
		\begin{align}\label{4.7}
		\|\v\|_{\mathrm{H}^{1}(0,\tau;\mathbb{H}_{0}^{1}(\mathcal{O}) \cap \mathbb{H}^{2}(\mathcal{O}))} \leq C\big(\| \v_{0}\|_{\mathbb{H}_{0}^{1}(\mathcal{O}) \cap \mathbb{H}^{2}(\mathcal{O})}+ \|G\|_{\mathrm{L}^{2}\left(0,\tau;\mathbb{L}^{2}(\mathcal{O})\right)}\big),
		\end{align}
		where
		\begin{align}\label{G}
		G&:= \underbrace{k\Delta \u_{0}}_{:=G_1}\underbrace{+\int_{0}^{t}k(t-s)\Delta \tilde{\v}(s)\d s}_{:=G_2}\underbrace{-(\tilde{\v}\cdot \nabla)\left(\u_{ 0}+\int_{0}^{t}\tilde{\v}(s)\d s\right)}_{:=G_3}\nonumber\\&\quad\underbrace{-\left(\left(\u_{0}+\int_{0}^{t}\tilde{\v}(s)\d s\right) \cdot \nabla\right)\tilde{\v}}_{:=G_4}.
		\end{align} 
		Next, we estimate each $G_{i}$'s $(i=1,2,3,4)$ separately as follows:
		\begin{align}\label{G1}
		\|G_{1}\|_{\mathrm{L}^{2}(0,\tau;\mathbb{L}^{2}(\mathcal{O}))}
		= \| k\|_{\mathrm{L}^{2}(0,\tau)}\|\Delta \u_{0}\|_{\mathbb{L}^{2}(\mathcal{O})}\leq C\| k\|_{\mathrm{L}^{2}(0,\tau)}\| \u_{0}\|_{\mathbb{H}^{2}(\mathcal{O})}.
		\end{align}
		Applying Lemma \ref{Lemma1}, we obtain
		\begin{align}\label{G2}
		\|G_{2}\|_{\mathrm{L}^{2}(0,\tau;\mathbb{L}^{2}(\mathcal{O}))}
		\leq \| k\|_{\mathrm{L}^{1}(0,\tau)}
		\|\Delta \tilde{\v}\|_{\mathrm{L}^{2}(0,\tau;\mathbb{L}^{2}(\mathcal{O}))}\leq C\tau^{1/2}\| k\|_{\mathrm{L}^{2}(0,\tau)}\|\tilde{\v}\|_{\mathrm{L}^{2}(0,\tau;\mathbb{H}^{2}(\mathcal{O}))}.
		\end{align}
		In order to estimate $G_{3}$, we split it into two parts
		\begin{align}\label{G3}
		G_{3}=\left(\tilde{\v}\cdot \nabla\right)\u_{0}+\left(\tilde{\v}\cdot \nabla\right)\left(\int_{0}^{t}\tilde{\v}(s)\d s\right)
		=:G_{31}+G_{32}.
		\end{align}
		Applying  H\"older's inequality, Sobolev embedding theorem and Gagliardo-Nirenberg's inequality, we have
		\begin{align}\label{G31}
		\|G_{31}\|_{\mathrm{L}^{2}\left(0,\tau;\mathbb{L}^{2}(\mathcal{O})\right)}&=\left(\int_{0}^{\tau}\|(\tilde{\v}\cdot \nabla)\u_{0}\|_{\mathbb{L}^{2}(\mathcal{O})}^{2} \d s\right)^{1/2} \nonumber\\& \leq  \|\nabla \u_{0}\|_{\mathbb{L}^{6}(\mathcal{O})}\bigg(\int_{0}^{\tau}\|\tilde{\v}(s)\|_{\mathbb{L}^{3}(\mathcal{O})}^{2}\d s\bigg)^{1/2} \nonumber\\
		&\leq C\| 
		\u_{0}\|_{\mathbb{H}^{2}(\mathcal{O})} \big\|\tilde{\v}\big\|^{1/4}_{\mathrm{L}^{2}\left(0,\tau;\mathbb{H}^{2}(\mathcal{O})\right)}\big\|\tilde{\v}\big\|^{3/4}_{\mathrm{L}^{2}\left(0,\tau;\mathbb{L}^{2}(\mathcal{O})\right)}.
		\end{align}
		Similarly for $G_{32}$, we find
		\begin{align}\label{G32}
		\|G_{32}\|_{\mathrm{L}^{2}(0,\tau;\mathbb{L}^{2}(\mathcal{O}))}&\leq \left\|\|\tilde{\v}\|_{\mathbb{L}^{3}(\mathcal{O})} \left\|\int_{0}^{t}\nabla \tilde{\v}(s)\d s\right\|_{\mathbb{L}^{6}(\mathcal{O})}\right\|_{\mathrm{L}^{2}(0,\tau)}\nonumber
		\\&\leq \|\tilde{\v}\|_{\mathrm{L}^{2}(0,\tau;\mathbb{L}^{3}(\mathcal{O}))} \left\|\int_{0}^{t}\nabla \tilde{\v}(s)\d s\right\|_{\mathrm{L}^{\infty}(0,\tau;\mathbb{L}^{6}(\mathcal{O}))}\nonumber
		\\& \leq C \|\tilde{\v}\|_{\mathrm{L}^{2}(0,\tau;\mathbb{H}^{2}(\mathcal{O}))}^{1/4}\|\tilde{\v}\|_{\mathrm{L}^{2}(0,\tau;\mathbb{L}^{2}(\mathcal{O}))}^{3/4}\int_{0}^{\tau}\|\tilde{\v}(s)\|_{\mathbb{H}^{2}(\mathcal{O})}\d s \nonumber
		\\ &\leq C\tau^{1/2}\|\tilde{\v}\|_{\mathrm{L}^{2}(0,\tau;\mathbb{H}^{2}(\mathcal{O}))}^{5/4}\|\tilde{\v}\|_{\mathrm{L}^{2}(0,\tau;\mathbb{L}^{2}(\mathcal{O}))}^{3/4}.
		\end{align}
		Finally, we estimate $G_4$ as 
		\begin{align}\label{G4}
		\|G_{4}\|_{\mathrm{L}^{2}(0,\tau;\mathbb{L}^{2}(\mathcal{O}))}& \leq \left\|\left\| \u_{0}+\int_{0}^{t}\tilde{\v}(s)\d s\right\|_{\mathbb{L}^{\infty}(\mathcal{O})}\left\|\nabla \tilde{\v}\right\|_{\mathbb{L}^{2}(\mathcal{O})}\right\|_{\mathrm{L}^{2}(0,\tau)}\nonumber\\
		& \leq \left\| \u_{0}+\int_{0}^{t}\tilde{\v}(s)\d s\right\|_{\mathrm{L}^{\infty}(0,\tau;\mathbb{L}^{\infty}(\mathcal{O}))}\left\|\nabla \tilde{\v}\right\|_{\mathrm{L}^{2}(0,\tau;\mathbb{L}^{2}(\mathcal{O}))}\nonumber\\
		&\leq C\left(\| \u_{0}\|_{\mathbb{H}^{2}(\mathcal{O})}+\int_{0}^{\tau}\|\tilde{\v}(s)\|_{\mathbb{H}^{2}(\mathcal{O})}\d s\right)\|\nabla \tilde{\v}\|_{\mathrm{L}^{2}(0,\tau;\mathbb{L}^{2}(\mathcal{O}))}\nonumber\\
		&\leq C\left(\| \u_{0}\|_{\mathbb{H}^{2}(\mathcal{O})}+\tau^{1/2}\|\tilde{\v}\|_{\mathrm{L}^{2}\left(0,\tau;\mathbb{H}^{2}(\mathcal{O})\right)}\right)\| \tilde{\v}\|_{\mathrm{L}^{2}(0,\tau;\mathbb{H}^{2}(\mathcal{O}))}^{1/2}\| \tilde{\v}\|_{\mathrm{L}^{2}(0,\tau;\mathbb{L}^{2}(\mathcal{O}))}^{1/2}.
		\end{align}
		Combining  \eqref{G}-\eqref{G4} and making use of \eqref{4.5} and definition of the space $\mathcal{V}(\tau,L)$, we obtain the following estimate
		\begin{align}\label{4.8}
		& \|G\|_{\mathrm{L}^{2}(0,\tau;\mathbb{L}^{2}(\mathcal{O}))} \nonumber\\&\leq C \bigg\{\left(\| \u_{0}\|_{\mathbb{H}^{2}(\mathcal{O})}+\tau^{1/2}\|\tilde{\v}\|_{\mathrm{L}^{2}(0,\tau;\mathbb{H}^{2}(\mathcal{O}))}\right)\| k\|_{\mathrm{L}^{2}(0,\tau)}+\| \u_{0}\|_{\mathbb{H}^{2}(\mathcal{O})}\big\|\tilde{\v}\big\|^{1/4}_{\mathrm{L}^{2}\left(0,\tau;\mathbb{H}^{2}(\mathcal{O})\right)}\|\tilde{\v}\|^{3/4}_{\mathrm{L}^{2}\left(0,\tau;\mathbb{L}^{2}(\mathcal{O})\right)}\nonumber\\&\quad+\tau^{1/2}\big\|\tilde{\v}\big\|^{5/4}_{\mathrm{L}^{2}\left(0,\tau;\mathbb{H}^{2}(\mathcal{O})\right)}\|\tilde{\v}\|^{3/4}_{\mathrm{L}^{2}\left(0,\tau;\mathbb{L}^{2}(\mathcal{O})\right)}\nonumber+\left(\| \u_{0}\|_{\mathbb{H}^{2}(\mathcal{O})}+\tau^{1/2}\|\tilde{\v}\|_{\mathrm{L}^{2}(0,\tau;\mathbb{H}^{2}(\mathcal{O}))}\right)\\&\quad \times\|\tilde{\v}\|_{\mathrm{L}^{2}(0,\tau;\mathbb{H}^{2}(\mathcal{O}))}^{1/2}\| \tilde{\v}\|_{\mathrm{L}^{2}(0,\tau;\mathbb{L}^{2}(\mathcal{O}))}^{1/2}\bigg\}
		\nonumber\\
		&\leq C\bigg\{ \left(\| \u_{0}\|_{\mathbb{H}^{2}(\mathcal{O})}+\tau^{1/2}L\right)\|k\|_{\mathrm{L}^{2}(0,\tau)}+ (\tau^{3/8}L^{1/4}\|\u_{0}\|_{\mathbb{H}^{2}(\mathcal{O})}+\tau^{7/8}L^{5/4})\left(\|\v_{0}\|_{\mathbb{L}^{2}(\mathcal{O})}+\tau^{1/2}L\right)^{3/4}\nonumber\\&\quad+\tau^{1/4}L^{1/2}\left(\|\u_{0}\|_{\mathbb{H}^{2}(\mathcal{O})}+\tau^{1/2}L\right)\left(\|\v_{0}\|_{\mathbb{L}^{2}(\mathcal{O})}+\tau^{1/2}L\right)^{1/2}\bigg\}.
		\end{align}
		If we take $\tau >0$ small enough such that 
		\begin{align}\label{4.9}
		\tau(1+L+L^{2}) \leq 1 ,
		\end{align}
		estimates \eqref{4.6}-\eqref{4.8}, transform to
		\begin{align*}
		&\|\v\|_{\mathrm{H}^{1}(0,\tau;\mathbb{H}_{0}^{1}(\mathcal{O}) \cap \mathbb{H}^{2}(\mathcal{O}))} +\|k\|_{\mathrm{L}^{2}(0,\tau)}\\
		&\leq C\big(\|\v_{0}\|_{\mathbb{H}_{0}^{1}(\mathcal{O}) \cap \mathbb{H}^{2}(\mathcal{O})}+\|G\|_{\mathrm{L}^{2}(0,\tau;\mathbb{L}^{2}(\mathcal{O}))}\big) +\|k\|_{\mathrm{L}^{2}(0,\tau)}\\
		&\leq C\bigg\{\|\v_{0}\|_{\mathbb{H}_{0}^{1}(\mathcal{O}) \cap \mathbb{H}^{2}(\mathcal{O})}+\left(\| \u_{0}\|_{\mathbb{H}^{2}(\mathcal{O})}+1\right)\|k\|_{\mathrm{L}^{2}(0,\tau)}+(\|\u_{0}\|_{\mathbb{H}^{2}(\mathcal{O})}+1)\left(\|\v_{0}\|_{\mathbb{L}^{2}(\mathcal{O})}+1\right)^{3/4}\\&\quad+\left(\|\u_{0}\|_{\mathbb{H}^{2}(\mathcal{O})}+1\right)\left(\|\v_{0}\|_{\mathbb{L}^{2}(\mathcal{O})}+1\right)^{1/2}\bigg\}\\
		&\leq C\bigg[\|\v_{0}\|_{\mathbb{H}_{0}^{1}(\mathcal{O}) \cap \mathbb{H}^{2}(\mathcal{O})}+\alpha \left(\| \u_{0}\|_{\mathbb{H}^{2}(\mathcal{O})}+1\right)\bigg\{\|r''\|_{\mathrm{L}^{2}(0,\tau)}+\bigg[\mu_{0}\|\Delta \varphi\|_{\mathbb{L}^{2}(\mathcal{O})}+\|\Delta \varphi\|_{\mathbb{L}^{2}(\mathcal{O})}\\&\quad+\big(\|\u_{0}\|_{\mathbb{H}^{2}(\mathcal{O})}+1\big)\|\nabla \varphi \|_{\mathbb{L}^{2}(\mathcal{O})}\bigg] \left(\|\v_{0}\|_{\mathbb{L}^{2}(\mathcal{O})}+1\right)\bigg\}+ (\|\u_{0}\|_{\mathbb{H}^{2}(\mathcal{O})}+1)\left(\|\v_{0}\|_{\mathbb{L}^{2}(\mathcal{O})}+1\right)^{3/4}\nonumber\\&\quad+\left(\|\u_{0}\|_{\mathbb{H}^{2}(\mathcal{O})}+1\right)\left(\|\v_{0}\|_{\mathbb{L}^{2}(\mathcal{O})}+1\right)^{1/2}\bigg]=:L_{0}.
		\end{align*}
		We notice that $L_{0}=L_{0}(\tau)$ is non-decreasing. Therefore, we can define a mapping $ \Psi : \mathcal{V}(\tau,L) \to \mathcal{V}(\tau,L)$ by choosing $L \geq L_{0}$ and $\tau$ small enough as in \eqref{4.9}.
		\vskip2mm
		\noindent
		\textbf{Step 2}. In this step, we prove that $\Psi$ is a contraction mapping.
		For $j \in \{1,2\}$, let $(\tilde{\v}_{j}, \tilde{k}_{j}) \in \mathcal{V}(\tau,L) $ and define $k_{j}$ and $(\v_{j}, p_{j})$ by \eqref{4.1} and \eqref{4.2} with $(\tilde{\v}, \tilde{k})=(\tilde{\v}_{j}, \tilde{k}_{j})$, respectively. Then, 
		$(\v_{1}, p_{1}, k_{1})$ and $(\v_{2}, p_{2}, k_{2})$ satisfies
		\begin{equation}\label{4.10}
		\left\{
		\begin{aligned}
		\partial_{t}{(\v_{1}-\v_{2})}-\mu_{1} \partial_{t} \Delta (\v_{1}-\v_{2})-&\mu_{0} \Delta (\v_{1}-\v_{2}) - (k_{1}-k_{2})\Delta \u_{0}\\-\int_{0}^{t}(k_{1}-k_{2})(t-s) \Delta \tilde{\v}_{1}(s)\d s&-\int_{0}^{t}k_{2}(t-s) \Delta (\tilde{\v}_{1}-\tilde{\v}_{2})(s)\d s\\+\big((\tilde{\v}_{1}-\tilde{\v}_{2}) \cdot \nabla\big)\bigg(\u_{0}+\int_{0}^{t}\tilde {\v}_{1}(s)\d s\bigg)&+\big(\tilde{\v}_{2} \cdot \nabla\big)\bigg(\int_{0}^{t}(\tilde{\v}_{1}-\tilde{\v}_{2})(s)\d s\bigg)\\+\bigg(\int_{0}^{t}(\tilde{\v}_{1}-\tilde{\v}_{2})(s)\d s \cdot \nabla\bigg)\tilde{\v}_{1}&+\left[\left(\u_{0}+\int_{0}^{t}\tilde{\v}_{2}(s)\d s\right) \cdot \nabla\right](\tilde{\v}_{1}-\tilde{\v}_{2})\\+\nabla \partial_{t}{(p_{1}-p_{2})}&=\mathbf{0}, \ \ \mbox{ in }\ \mathcal{O} \times (0,\tau), \\
		\nabla \cdot (\v_{1}-\v_{2})&=0,  \ \  \mbox{ in }\ \mathcal{O} \times (0,\tau), \\
		\v_{1}-\v_{2}&=\mathbf{0}, \ \  \mbox{ on }\ \partial \mathcal{O} \times [0,\tau), \\
		\v_{1}-\v_{2}&=\mathbf{0}, \ \  \mbox{ in }\ \mathcal{O} \times \{0\}
		\end{aligned}
		\right.
		\end{equation}
		and
		\begin{equation}\label{4.11}
		\begin{aligned}
		&(k_{1}-k_{2})(t)\\=&-\alpha \bigg\{\underbrace{\mu_{0} \int_{\mathcal{O}}(\tilde{\v}_{1}-\tilde{\v}_{2}) \cdot \Delta \varphi \d x}_{:=B_1} +\underbrace{\int_{0}^{t}\int_{\mathcal{O}}(\tilde{k}_{1}-\tilde{k}_{2})(t-s) \tilde{\v}_{1} \cdot\Delta \varphi \d x \d s}_{:=B_2}\\&+\underbrace{\int_{0}^{t}\int_{\mathcal{O}}\tilde{k}_{2}(t-s)(\tilde{\v}_{1}-\tilde{\v}_{2}) \cdot\Delta \varphi \d x \d s}_{:=B_3}+\underbrace{\int_{\mathcal{O}}\big[\big((\tilde{\v}_{1}-\tilde{\v}_{2})\cdot \nabla\big)\varphi\big]\cdot\left( \u_{0}+\int_{0}^{t}\tilde{\v}_{1}(s)\d s\right)\d x}_{:=B_4}\\&+\underbrace{\int_{\mathcal{O}}\big((\tilde{\v}_{2}\cdot \nabla)\varphi\big) \cdot \left(\int_{0}^{t}(\tilde{\v}_{1}-\tilde{\v}_{2})(s)\d s\right)\d x}_{:=B_5}+\underbrace{\int_{\mathcal{O}}\left[\left(\int_{0}^{t}(\tilde{\v}_{1}-\tilde{\v}_{2})(s)\d s \cdot \nabla\right)\varphi\right]\cdot \tilde{\v}_{1}\d x}_{:=B_6}\\&+\underbrace{\int_{\mathcal{O}}\left[\left(\left(\u_{0}+\int_{0}^{t}\tilde{\v}_{2}(s)\d s\right) \cdot \nabla\right) \varphi\right]\cdot (\tilde{\v}_{1}-\tilde{\v}_{2})\d x}_{:=B_7}\bigg\}.
		\end{aligned}
		\end{equation}
		Let us now estimate $(k_{1}-k_{2})$ from the equation \eqref{4.11}. For this, our aim is to bound each individual term of $(k_{1}-k_{2})$, that is, we estimate $B_{i}\text{'s}, \ i=1,\ldots,7$.
		Applying H\"older's inequality, we have
		\begin{align*}
		\|B_{1}\|_{\mathrm{L}^{2}(0,\tau)}&\leq \mu_{0} \left(\int_{0}^{\tau}\|(\tilde{\v}_{1}-\tilde{\v}_{2})(s)\|_{\mathbb{L}^{2}(\mathcal{O})}^{2} \| \Delta \varphi \|_{\mathbb{L}^{2}(\mathcal{O})}^{2}\d s\right)^{1/2}\\
		&= \mu_{0} \| \tilde{\v}_{1}-\tilde{\v}_{2}\|_{\mathrm{L}^{2}(0,\tau;\mathbb{L}^{2}(\mathcal{O}))} \| \Delta \varphi \|_{\mathbb{L}^{2}(\mathcal{O})}.
		\end{align*}
		Using Lemma \ref{Lemma1}, we obtain
		\begin{align*}
		\|B_{2}\|_{\mathrm{L}^{2}(0,\tau)} &\leq \|\tilde{k}_{1}-\tilde{k}_{2}\|_{\mathrm{L}^{1}(0,\tau)} \left\|\int_{\mathcal{O}} \tilde{\v}_{1} \cdot\Delta \varphi \d x\right\|_{\mathrm{L}^{2}(0,\tau)}\\
		&\leq \tau^{1/2}\|\tilde{k}_{1}-\tilde{k}_{2}\|_{\mathrm{L}^{2}(0,\tau)} \| \tilde{\v}_{1}\|_{\mathrm{L}^{2}(0,\tau;\mathbb{L}^{2}(\mathcal{O}))} \| \Delta \varphi \|_{\mathbb{L}^{2}(\mathcal{O})},
		\end{align*}
		and 
		\begin{align*}
		\|B_{3}\|_{\mathrm{L}^{2}(0,\tau)}\leq \tau^{1/2}\|\tilde{k}_{2}\|_{\mathrm{L}^{2}(0,\tau)} \| \tilde{\v}_{1}-\tilde{\v}_{2}\|_{\mathrm{L}^{2}(0,\tau;\mathbb{L}^{2}(\mathcal{O}))} \| \Delta \varphi \|_{\mathbb{L}^{2}(\mathcal{O})}.
		\end{align*}
		Using H\"older's inequality and Sobolev embedding theorem, we have
		\begin{align*}
		\|B_{4}\|_{\mathrm{L}^{2}(0,\tau)}&\leq  \bigg\|\|\tilde{\v}_{1}-\tilde{\v}_{2}\|_{\mathbb{L}^{2}(\mathcal{O})}\| \nabla\varphi\|_{\mathbb{L}^{2}(\mathcal{O})} \left\|\u_{0}+\int_{0}^{t}\tilde{\v}_{1}(s)\d s\right\|_{\mathbb{L}^{\infty}(\mathcal{O})}\bigg\|_{\mathrm{L}^{2}(0,\tau)}\\
		&\leq  \|\tilde{\v}_{1}-\tilde{\v}_{2}\|_{\mathrm{L}^{2}(0,\tau;\mathbb{L}^{2}(\mathcal{O}))}\| \nabla\varphi\|_{\mathbb{L}^{2}(\mathcal{O})} \left\|\u_{0}+\int_{0}^{t}\tilde{\v}_{1}(s)\d s\right\|_{\mathrm{L}^{\infty}(0,\tau;\mathbb{L}^{\infty}(\mathcal{O}))}\\
		&\leq C\left(\|\u_{0}\|_{\mathbb{H}^{2}(\mathcal{O})}+\int_{0}^{\tau}\|\tilde{\v}_{1}(s)\|_{\mathbb{H}^{2}(\mathcal{O})}\d s\right)\|\tilde{\v}_{1}-\tilde{\v}_{2}\|_{\mathrm{L}^{2}(0,\tau;\mathbb{L}^{2}(\mathcal{O}))}\| \nabla\varphi\|_{\mathbb{L}^{2}(\mathcal{O})}\\
		&\leq C\big(\|\u_{0}\|_{\mathbb{H}^{2}(\mathcal{O})}+\tau^{1/2}\|\tilde{\v}_{1}\|_{\mathrm{L}^{2}(0,\tau;\mathbb{H}^{2}(\mathcal{O}))}\big)\|\tilde{\v}_{1}-\tilde{\v}_{2}\|_{\mathrm{L}^{2}(0,\tau;\mathbb{L}^{2}(\mathcal{O}))}\| \nabla\varphi\|_{\mathbb{L}^{2}(\mathcal{O})}
		\end{align*}
		and
		\begin{align*}
		\|B_{5}\|_{\mathrm{L}^{2}(0,\tau)} 
		&\leq  \bigg\|\|\tilde{\v}_{2}\|_{\mathbb{L}^{2}(\mathcal{O})}\| \nabla\varphi\|_{\mathbb{L}^{2}(\mathcal{O})} \left\|\int_{0}^{t}(\tilde{\v}_{1}-\tilde{\v}_{2})(s)\d s\right\|_{\mathbb{L}^{\infty}(\mathcal{O})}\bigg\|_{\mathrm{L}^{2}(0,\tau)}\\
		&\leq \|\tilde{\v}_{2}\|_{\mathrm{L}^{2}(0,\tau;\mathbb{L}^{2}(\mathcal{O}))} \|\nabla \varphi\|_{\mathbb{L}^{2}(\mathcal{O})}\left\|\int_{0}^{t}(\tilde{\v}_{1}-\tilde{\v}_{2})(s)\d s\right\|_{\mathrm{L}^{\infty}(0,\tau;\mathbb{L}^{\infty}(\mathcal{O}))}\\
		&\leq \left(\int_{0}^{\tau}\|(\tilde{\v}_{1}-\tilde{\v}_{2})(s)\|_{\mathbb{H}^{2}(\mathcal{O})}\d s\right)
		\|\tilde{\v}_{2}\|_{\mathrm{L}^{2}(0,\tau;\mathbb{L}^{2}(\mathcal{O}))}\| \nabla\varphi\|_{\mathbb{L}^{2}(\mathcal{O})}\\
		&\leq C\tau^{1/2}\|\tilde{\v}_{1}-\tilde{\v}_{2}\|_{\mathrm{L}^{2}(0,\tau;\mathbb{H}^{2}(\mathcal{O}))}\|\tilde{\v}_{2}\|_{\mathrm{L}^{2}(0,\tau;\mathbb{L}^{2}(\mathcal{O}))}\| \nabla\varphi\|_{\mathbb{L}^{2}(\mathcal{O})}.
		\end{align*}
		Applying similar argument as in the estimate of $B_5$, we get 
		\begin{align*}
		\|B_{6}\|_{\mathrm{L}^{2}(0,\tau)} 
		&\leq \left\|\left\|\int_{0}^{t}(\tilde{\v}_{1}-\tilde{\v}_{2})(s)\d s\right\|_{\mathbb{L}^{\infty}(\mathcal{O})}\|\nabla\varphi\|_{\mathbb{L}^{2}(\mathcal{O})}\|\tilde{\v}_{1}\|_{\mathbb{L}^{2}(\mathcal{O})}\right\|_{\mathrm{L}^{2}(0,\tau)} \\
		&\leq C\tau^{1/2} \|\tilde{\v}_{1}-\tilde{\v}_{2}\|_{\mathrm{L}^{2}(0,\tau;\mathbb{H}^{2}(\mathcal{O}))}\|\nabla\varphi\|_{\mathbb{L}^{2}(\mathcal{O})}\|\tilde{\v}_{1}\|_{\mathrm{L}^{2}(0,\tau;\mathbb{L}^{2}(\mathcal{O}))}.
		\end{align*}
		Applying similarly argument as for $B_4$, we find
		\begin{align*}
		\|B_{7}\|_{\mathrm{L}^{2}(0,\tau)}&\leq
		C\left(\|\u_{0}\|_{\mathbb{H}^{2}(\mathcal{O})}+\tau^{1/2}\|\tilde{\v}_{2}\|_{\mathrm{L}^{2}(0,\tau;\mathbb{H}^{2}(\mathcal{O}))}\right)\|\nabla\varphi\|_{\mathbb{L}^{2}(\mathcal{O})}\|\tilde{\v}_{1}-\tilde{\v}_{2}\|_{\mathrm{L}^{2}(0,\tau;\mathbb{L}^{2}(\mathcal{O}))}.
		\end{align*}
		Combining the estimates for $B_{i}$'s, one can conclude that
		\begin{align*}
		&\|k_{1}-k_{2}\|_{\mathrm{L}^{2}(0,\tau)} \\&\leq \alpha \bigg\{\mu_{0}\| \Delta \varphi \|_{\mathbb{L}^{2}(\mathcal{O})} \| \tilde{\v}_{1}-\tilde{\v}_{2}\|_{\mathrm{L}^{2}(0,\tau;\mathbb{L}^{2}(\mathcal{O}))}+\tau^{1/2}\| \Delta \varphi \|_{\mathbb{L}^{2}(\mathcal{O})}\|\tilde{k}_{1}-\tilde{k}_{2}\|_{\mathrm{L}^{2}(0,\tau)} \| \tilde{\v}_{1}\|_{\mathrm{L}^{2}(0,\tau;\mathbb{L}^{2}(\mathcal{O}))} \\&\quad+\tau^{1/2}\| \Delta \varphi \|_{\mathbb{L}^{2}(\mathcal{O})}\|\tilde{k}_{2}\|_{\mathrm{L}^{2}(0,\tau)} \| \tilde{\v}_{1}-\tilde{\v}_{2}\|_{\mathrm{L}^{2}(0,\tau;\mathbb{L}^{2}(\mathcal{O}))} \nonumber\\&\quad+C\bigg[2\|\u_{0}\|_{\mathbb{H}^{2}(\mathcal{O})}+\tau^{1/2}\bigg(\|\tilde{\v}_{1}\|_{\mathrm{L}^{2}(0,\tau;\mathbb{H}^{2}(\mathcal{O}))}+\|\tilde{\v}_{2}\|_{\mathrm{L}^{2}(0,\tau;\mathbb{H}^{2}(\mathcal{O}))}\bigg)\bigg]\nonumber\\&\qquad\times\| \nabla\varphi\|_{\mathbb{L}^{2}(\mathcal{O})}\|\tilde{\v}_{1}-\tilde{\v}_{2}\|_{\mathrm{L}^{2}(0,\tau;\mathbb{L}^{2}(\mathcal{O}))}\\&\quad+C\tau^{1/2}\bigg(\|\tilde{\v}_{1}\|_{\mathrm{L}^{2}(0,\tau;\mathbb{L}^{2}(\mathcal{O}))}+\|\tilde{\v}_{2}\|_{\mathrm{L}^{2}(0,\tau;\mathbb{L}^{2}(\mathcal{O}))}\bigg)\| \nabla\varphi\|_{\mathbb{L}^{2}(\mathcal{O})}\|\tilde{\v}_{1}-\tilde{\v}_{2}\|_{\mathrm{L}^{2}(0,\tau;\mathbb{H}^{2}(\mathcal{O}))}\bigg\}.
		\end{align*}
		Using Lemma \ref{Lemma2} and the definition of the space $\mathcal{V}(\tau,L)$, we obtain the following estimate
		\begin{align}\label{4.12}
		&\|k_{1}-k_{2}\|_{\mathrm{L}^{2}(0,\tau)}\nonumber  
		\\&\leq \alpha \bigg\{ \mu_{0}\tau \| \Delta \varphi \|_{\mathbb{L}^{2}(\mathcal{O})}\| \partial_{t}(\tilde{\v}_{1}-\tilde{\v}_{2})\|_{\mathrm{L}^{2}(0,\tau;\mathbb{L}^{2}(\mathcal{O}))} +\tau^{1/2}L\| \Delta \varphi \|_{\mathbb{L}^{2}(\mathcal{O})}\|\tilde{k}_{1}-\tilde{k}_{2}\|_{\mathrm{L}^{2}(0,\tau)} \nonumber \\&\quad+\tau^{3/2}L\| \Delta \varphi \|_{\mathbb{L}^{2}(\mathcal{O})} \|\partial_{t} (\tilde{\v}_{1}-\tilde{\v}_{2})\|_{\mathrm{L}^{2}(0,\tau;\mathbb{L}^{2}(\mathcal{O}))} \nonumber\\&\quad +C\bigg(2\|\u_{0}\|_{\mathbb{H}^{2}(\mathcal{O})}+2\tau^{1/2}L\bigg) \| \nabla\varphi\|_{\mathbb{L}^{2}(\mathcal{O})} \tau\|\partial_{t}(\tilde{\v}_{1}-\tilde{\v}_{2})\|_{\mathrm{L}^{2}(0,\tau;\mathbb{L}^{2}(\mathcal{O}))}\nonumber\\&\quad+C\tau^{1/2}L\| \nabla\varphi\|_{\mathbb{L}^{2}(\mathcal{O})}\|\tilde{\v}_{1}-\tilde{\v}_{2}\|_{\mathrm{L}^{2}(0,\tau;\mathbb{H}^{2}(\mathcal{O}))}\bigg\}	\nonumber\\&\leq \nonumber\alpha \tau^{1/2} \bigg\{\bigg(\mu_{0} \tau^{1/2}  + \tau L\bigg) \| \Delta \varphi \|_{\mathbb{L}^{2}(\mathcal{O})}+C\bigg(\tau^{1/2}\|\u_{0}\|_{\mathbb{H}^{2}(\mathcal{O})}+\tau L+L\bigg) \nonumber\| \nabla\varphi\|_{\mathbb{L}^{2}(\mathcal{O})} \bigg\}\nonumber\\&\quad\times\| \tilde{\v}_{1}-\tilde{\v}_{2}\|_{\mathrm{H}^{1}(0,\tau;\mathbb{H}_{0}^{1}(\mathcal{O}) \cap \mathbb{H}^{2}(\mathcal{O}))} +\alpha\tau^{1/2}L\| \Delta \varphi \|_{\mathbb{L}^{2}(\mathcal{O})}\|\tilde{k}_{1}-\tilde{k}_{2}\|_{\mathrm{L}^{2}(0,\tau)}.
		\end{align}
		Again, using the estimates for the linear problem \eqref{4.10}, we estimate
		\begin{align}\label{4.13}
		\|\v_{1}-\v_{2}\|_{\mathrm{H}^{1}(0,\tau;\mathbb{H}_{0}^{1}(\mathcal{O}) \cap \mathbb{H}^{2}(\mathcal{O}))} \leq C \|E\|_{\mathrm{L}^{2}(0,\tau;\mathbb{L}^{2}(\mathcal{O}))},
		\end{align}
		where
		\begin{align*}
		E&:= \underbrace{(k_{1}-k_{2})\Delta \u_{0}}_{:=E_1}+\underbrace{\int_{0}^{t}(k_{1}-k_{2})(t-s) \Delta \tilde{\v}_{1}(s)\d s}_{:=E_2}+\underbrace{\int_{0}^{t}k_{2}(t-s) \Delta (\tilde{\v}_{1}-\tilde{\v}_{2})(s)\d s}_{E_3}\\&\quad\underbrace{-\big((\tilde{\v}_{1}-\tilde{\v}_{2}) \cdot \nabla\big)\left(\u_{0}+\int_{0}^{t}\tilde {\v}_{1}(s)\d s\right)}_{:=E_4}\underbrace{-\big(\tilde{\v}_{2} \cdot \nabla\big)\left(\int_{0}^{t}(\tilde{\v}_{1}-\tilde{\v}_{2})(s)\d s\right)}_{:=E_5}\\&\quad\underbrace{-\left(\int_{0}^{t}(\tilde{\v}_{1}-\tilde{\v}_{2})(s)\d s \cdot \nabla\right)\tilde{\v}_{1}}_{:=E_6}\underbrace{-\left[ \left(\u_{0}+\int_{0}^{t}\tilde {\v}_{2}(s)\d s\right) \cdot \nabla\right](\tilde{\v}_{1}-\tilde{\v}_{2})}_{:=E_7} .
		\end{align*}
		For $E_{j} \ (j=1,2,\ldots,7)$, we obtain the following estimates:
		\begin{align*}
		\|E_{1}\|_{\mathrm{L}^{2}(0,\tau;\mathbb{L}^{2}(\mathcal{O}))}=\|k_{1}-k_{2}\|_{\mathrm{L}^{2}(0,\tau)}\|\Delta \u_{0}\|_{\mathbb{L}^{2}(\mathcal{O})}\leq C\|k_{1}-k_{2}\|_{\mathrm{L}^{2}(0,\tau)}\| \u_{0}\|_{\mathbb{H}^{2}(\mathcal{O})}.
		\end{align*}
		Using Lemma \ref{Lemma1}, we get
		\begin{align*}
		\|E_{2}\|_{\mathrm{L}^{2}(0,\tau;\mathbb{L}^{2}(\mathcal{O}))}\leq
		C\tau^{1/2}\|k_{1}-k_{2}\|_{\mathrm{L}^{2}(0,\tau)}\| \tilde{\v}_{1}\|_{\mathrm{L}^{2}(0,\tau;\mathbb{H}^{2}(\mathcal{O}))},
		\end{align*}
		and
		\begin{align*}
		\|E_{3}\|_{\mathrm{L}^{2}(0,\tau;\mathbb{L}^{2}(\mathcal{O}))}\leq C\tau^{1/2}\|k_{2}\|_{\mathrm{L}^{2}(0,\tau)}\| \tilde{\v}_{1}- \tilde{\v}_{2}\|_{\mathrm{L}^{2}(0,\tau;\mathbb{H}^{2}(\mathcal{O}))}. 
		\end{align*}
		By an argument similar to the estimate of $\|G_{3}\|_{\mathrm{L}^{2}(0,\tau;\mathbb{L}^{2}(\mathcal{O}))}$, we obtain
		\begin{align*}
		E_{4}= \underbrace{\big((\tilde{\v}_{1}-\tilde{\v}_{2}) \cdot \nabla\big)\u_{0}}_{:=E_{41}}+\underbrace{\big((\tilde{\v}_{1}-\tilde{\v}_{2}) \cdot \nabla\big)\int_{0}^{t}\tilde {\v}_{1}(s)\d s}_{:=E_{42}}
		\end{align*}
		\begin{align*}
		\|E_{41}\|_{\mathrm{L}^{2}(0,\tau;\mathbb{L}^{2}(\mathcal{O}))}
		&\leq C\|\u_{0}\|_{\mathbb{H}^{2}(\mathcal{O})}\|\tilde{\v}_{1}-\tilde{\v}_{2} \|_{\mathrm{L}^{2}(0,\tau;\mathbb{H}^{2}(\mathcal{O}))}^{1/4}
		\|\tilde{\v}_{1}-\tilde{\v}_{2} \|_{\mathrm{L}^{2}\left(0,\tau;\mathbb{L}^{2}(\mathcal{O})\right)}^{3/4},\\
		\|E_{42}\|_{\mathrm{L}^{2}(0,\tau;\mathbb{L}^{2}(\mathcal{O}))}\nonumber&\leq \left\|\|\tilde{\v}_{1}-\tilde{\v}_{2}\|_{\mathbb{L}^{3}(\mathcal{O})}\left\|\int_{0}^{t}\nabla\tilde{\v}_{1}(s)\d s\right\|_{\mathbb{L}^{6}(\mathcal{O})}\right\|_{\mathrm{L}^{2}(0,\tau)}
		\\& \leq
		\|\tilde{\v}_{1}-\tilde{\v}_{2}\|_{\mathrm{L}^{2}(0,\tau;\mathbb{L}^{3}(\mathcal{O}))}\left\|\int_{0}^{t}\nabla\tilde{\v}_{1}(s)\d s\right\|_{\mathrm{L}^{\infty}(0,\tau;\mathbb{L}^{6}(\mathcal{O}))}\\&\leq C \|\tilde{\v}_{1}-\tilde{\v}_{2}\|_{\mathrm{L}^{2}(0,\tau;\mathbb{H}^{2}(\mathcal{O}))}^{1/4} \|\tilde{\v}_{1}-\tilde{\v}_{2}\|_{\mathrm{L}^{2}(0,\tau;\mathbb{L}^{2}(\mathcal{O}))}^{3/4}\int_{0}^{\tau}\| \tilde{\v}_{1}(s)\|_{\mathbb{H}^{2}(\mathcal{O})}\d s\nonumber\\
		&\leq C\tau^{1/2}\|\tilde{\v}_{1}\|_{\mathrm{L}^{2}(0,\tau;\mathbb{H}^{2}(\mathcal{O}))}\|\tilde{\v}_{1}-\tilde{\v}_{2}\|_{\mathrm{L}^{2}(0,\tau;\mathbb{H}^{2}(\mathcal{O}))}^{1/4} \|\tilde{\v}_{1}-\tilde{\v}_{2}\|_{\mathrm{L}^{2}(0,\tau;\mathbb{L}^{2}(\mathcal{O}))}^{3/4}.
		\end{align*}
		Also, we have 
		\begin{align*}
		\|E_{5}\|_{\mathrm{L}^{2}(0,\tau;\mathbb{L}^{2}(\mathcal{O}))}&\leq \left \|\|\tilde{\v}_{2}\|_{\mathbb{L}^{\infty}(\mathcal{O})}\left\|\int_{0}^{t}\nabla(\tilde{\v}_{1}-\tilde{\v}_{2})(s)\d s\right\|_{\mathbb{L}^{2}(\mathcal{O})}\right \|_{\mathrm{L}^{2}(0,\tau)} \\ &\leq C \left\|\|\tilde{\v}_{2}\|_{\mathbb{H}^{2}(\mathcal{O})}\left\|\int_{0}^{t}\nabla(\tilde{\v}_{1}-\tilde{\v}_{2})(s)\d s\right\|_{\mathbb{L}^{2}(\mathcal{O})} \right\|_{\mathrm{L}^{2}(0,\tau)}\\              &\leq C\left( \int_{0}^{\tau}\|(\tilde{\v}_{1}-\tilde{\v}_{2})(s)\|_{\mathbb{H}^{1}_{0}(\mathcal{O})}\d s\right)\|\tilde{\v}_{2}\|_{\mathrm{L}^{2}(0,T;\mathbb{H}^{2}(\mathcal{O}))}\\
		&\leq C\tau^{1/2}\|\tilde{\v}_{1}-\tilde{\v}_{2}\|_{\mathrm{L}^{2}(0,\tau;\mathbb{H}^{1}_{0}(\mathcal{O}))}\|\tilde{\v}_{2}\|_{\mathrm{L}^{2}(0,\tau;\mathbb{H}^{2}(\mathcal{O}))},
		\end{align*}
		where we used H\"older's inequality and Sobolev embedding theorem. Similarly, we find
		\begin{align*}
		\|E_{6}\|_{\mathrm{L}^{2}(0,\tau;\mathbb{L}^{2}(\mathcal{O}))}
		&\leq
		\left\|\left\|\int_{0}^{t}(\tilde{\v}_{1}-\tilde{\v}_{2})(s)\d s\right\|_{\mathbb{L}^{\infty}(\mathcal{O})}\| \nabla\tilde{\v}_{1}\|_{\mathbb{L}^{2}(\mathcal{O})}\right\|_{\mathrm{L}^{2}(0,\tau)}\\
		&\leq \left(
		\int_{0}^{\tau}\|(\tilde{\v}_{1}-\tilde{\v}_{2})(s)\|_{\mathbb{L}^{\infty}(\mathcal{O})}\d s\right) \|\nabla\tilde{\v}_{1}\|_{\mathrm{L}^{2}(0,\tau;\mathbb{L}^{2}(\mathcal{O}))}\\
		&\leq C\left(
		\int_{0}^{\tau}\|(\tilde{\v}_{1}-\tilde{\v}_{2})(s)\|_{\mathbb{H}^{2}(\mathcal{O})}\d s\right) \|\tilde{\v}_{1}\|_{\mathrm{L}^{2}(0,\tau;\mathbb{H}_{0}^{1}(\mathcal{O}))}\\
		&\leq  C\tau^{1/2}\|\tilde{\v}_{1}-\tilde{\v}_{2}\|_{\mathrm{L}^{2}\left(0,\tau;\mathbb{H}^{2}(\mathcal{O})\right)}\|\tilde{\v}_{1}\|_{\mathrm{L}^{2}(0,\tau;\mathbb{H}^{1}_{0}(\mathcal{O}))}.
		\end{align*}
		Using H\"older's inequality, Sobolev embedding theorem and Gagliardo-Nirenberg's inequality, we get
		\begin{align*}
		\|E_{7}\|_{\mathrm{L}^{2}(0,\tau;\mathbb{L}^{2}(\mathcal{O}))}
		&\leq \left\|\left\|\u_{0}+\int_{0}^{t}\tilde {\v}_{2}(s)\d s\right\|_{\mathbb{L}^{\infty}(\mathcal{O})} 
		\|\nabla(\tilde{\v_{1}}-\tilde{\v_{2}})\|_{\mathbb{L}^{2}(\mathcal{O}))}\right\|_{\mathrm{L}^{2}(0,\tau)}\\
		&\leq \left(\|\u_{0}\|_{\mathbb{L}^{\infty}(\mathcal{O})} +\int_{0}^{\tau}\|\tilde {\v}_{2}(s)\|_{\mathrm{L}^{\infty}(\mathcal{O})}\d s\right) \|\nabla(\tilde{\v_{1}}-\tilde{\v_{2}})\|_{\mathrm{L}^{2}(0,\tau;\mathbb{L}^{2}(\mathcal{O}))}\\
		&\leq
		C\left(\|\u_{0}\|_{\mathbb{H}^{2}(\mathcal{O})}+\tau^{1/2}\|\tilde{\v}_{2}\|_{\mathrm{L}^{2}(0,\tau;\mathbb{H}^{2}(\mathcal{O}))}\right)\|\tilde{\v}_{1}-\tilde{\v}_{2}\|_{\mathrm{L}^{2}(0,\tau;\mathbb{H}^{2}(\mathcal{O}))}^{1/2} \\& \quad\times\|\tilde{\v}_{1}-\tilde{\v}_{2}\|_{\mathrm{L}^{2}(0,\tau;\mathbb{L}^{2}(\mathcal{O}))}^{1/2}.
		\end{align*}
		Combining the estimates for $E_{i}$'s, one can conclude that
		\begin{align*}
		&\|E\|_{\mathrm{L}^{2}(0,\tau;\mathbb{L}^{2}(\mathcal{O}))} \\&\leq C \bigg\{\big(\| \u_{0}\|_{\mathbb{H}^{2}(\mathcal{O})}+\tau^{1/2}\| \tilde{\v}_{1}\|_{\mathrm{L}^{2}(0,\tau;\mathbb{H}^{2}(\mathcal{O}))}\big)\|k_{1}-k_{2}\|_{\mathrm{L}^{2}(0,\tau)}
		+\tau^{1/2}\|k_{2}\|_{\mathrm{L}^{2}(0,\tau)}\| \tilde{\v}_{1}- \tilde{\v}_{2}\|_{\mathrm{L}^{2}(0,\tau;\mathbb{H}^{2}(\mathcal{O}))}\\&\quad+(\|\u_{0}\|_{\mathbb{H}^{2}(\mathcal{O})}+\tau^{1/2}\|\tilde{\v}_{1}\|_{\mathrm{L}^{2}(0,\tau;\mathbb{H}^{2}(\mathcal{O}))})\|\tilde{\v}_{1}-\tilde{\v}_{2} \|_{\mathrm{L}^{2}(0,\tau;\mathbb{H}^{2}(\mathcal{O}))}^{1/4}\|\tilde{\v}_{1}-\tilde{\v}_{2} \|_{\mathrm{L}^{2}(0,\tau;\mathbb{L}^{2}(\mathcal{O}))}^{3/4}\\&\quad+\tau^{1/2}\|\tilde{\v}_{1}-\tilde{\v}_{2}\|_{\mathrm{L}^{2}(0,\tau;\mathbb{H}^{1}_{0}(\mathcal{O}))}\|\tilde{\v}_{2}\|_{\mathrm{L}^{2}(0,\tau;\mathbb{H}^{2}(\mathcal{O}))}+\tau^{1/2}\|\tilde{\v}_{1}-\tilde{\v}_{2}\|_{\mathrm{L}^{2}(0,\tau;\mathbb{H}^{2}(\mathcal{O}))}\|\tilde{\v}_{1}\|_{\mathrm{L}^{2}(0,\tau;\mathbb{H}^{1}_{0}(\mathcal{O}))}\\&\quad+\big(\|\u_{0}\|_{\mathbb{H}^{2}(\mathcal{O})}+\tau^{1/2}\|\tilde{\v}_{2}\|_{\mathrm{L}^{2}(0,\tau;\mathbb{H}^{2}(\mathcal{O}))}\big)\|\tilde{\v}_{1}-\tilde{\v}_{2}\|_{\mathrm{L}^{2}(0,\tau;\mathbb{H}^{2}(\mathcal{O}))}^{1/2}  \|\tilde{\v}_{1}-\tilde{\v}_{2}\|_{\mathrm{L}^{2}(0,\tau;\mathbb{L}^{2}(\mathcal{O}))}^{1/2}\bigg\}.
		\end{align*}
		Using Lemma \ref{Lemma2} and the definition of the space $\mathcal{V}(\tau,L)$, we obtain the following estimate:
		\begin{align}\label{4.14}
		\nonumber &\|E\|_{\mathrm{L}^{2}(0,\tau;\mathbb{L}^{2}(\mathcal{O}))} \\&\leq C\bigg\{\big(\| \u_{0}\|_{\mathbb{H}^{2}(\mathcal{O})}\nonumber+\tau^{1/2}L\big)\|k_{1}-k_{2}\|_{\mathrm{L}^{2}(0,\tau)}
		+\tau^{1/2}L\|\tilde{\v}_{1}- \tilde{\v}_{2}\|_{\mathrm{H}^{1}(0,\tau;\mathbb{H}_{0}^{1}(\mathcal{O}) \cap \mathbb{H}^{2}(\mathcal{O}))}\\&\quad+\tau^{3/4}(\|\u_{0}\|_{\mathbb{H}^{2}(\mathcal{O})}+\tau^{1/2}L)\|\partial_{t}(\tilde{\v}_{1}-\tilde{\v}_{2}) \|_{\mathrm{L}^{2}(0,\tau;\mathbb{L}^{2}(\mathcal{O}))}^{3/4}\|\tilde{\v}_{1}-\tilde{\v}_{2} \|_{\mathrm{L}^{2}(0,\tau;\mathbb{H}^{2}(\mathcal{O}))}^{1/4}\nonumber\\&\quad+2\tau^{1/2}L\|\tilde{\v}_{1}- \tilde{\v}_{2}\|_{\mathrm{H}^{1}(0,\tau;\mathbb{H}_{0}^{1}(\mathcal{O}) \cap \mathbb{H}^{2}(\mathcal{O}))}\nonumber\\&\quad +\tau^{1/2}\big(\|\u_{0}\|_{\mathbb{H}^{2}(\mathcal{O})}+\tau^{1/2}L\big) \|\tilde{\v}_{1}-\tilde{\v}_{2}\|_{\mathrm{L}^{2}(0,\tau;\mathbb{H}^{2}(\mathcal{O}))}^{1/2}\|\partial_{t}(\tilde{\v}_{1}-\tilde{\v}_{2})\|_{\mathrm{L}^{2}(0,\tau;\mathbb{L}^{2}(\mathcal{O}))}^{1/2}\bigg\} \nonumber\\
		&\leq C\bigg\{\big(\| \u_{0}\|_{\mathbb{H}^{2}(\mathcal{O})}\nonumber+\tau^{1/2}L\big)\|k_{1}-k_{2}\|_{\mathrm{L}^{2}(0,\tau)}\\&\quad +\tau^{1/2}\bigg[3L+\big(\|\u_{0}\|_{\mathbb{H}^{2}(\mathcal{O})}+\tau^{1/2}L\big)(\tau^{1/4}+1)\bigg]\|\tilde{\v}_{1}-\tilde{\v}_{2}\|_{\mathrm{H}^{1}(0,\tau;\mathbb{H}_{0}^{1}(\mathcal{O}) \cap \mathbb{H}^{2}(\mathcal{O}))}\bigg\}.
		\end{align}
		Thus, combining \eqref{4.12}-\eqref{4.14}, we finally arrive at
		\begin{align*}
		&\|\v_{1}-\v_{2}\|_{\mathrm{H}^{1}(0,\tau;\mathbb{H}_{0}^{1}(\mathcal{O}) \cap \mathbb{H}^{2}(\mathcal{O}))}+\|k_{1}-k_{2}\|_{\mathrm{L}^{2}(0,\tau)} 
		\\&\leq C\|E\|_{\mathrm{L}^{2}(0,\tau;\mathbb{L}^{2}(\mathcal{O}))} +\|k_{1}-k_{2}\|_{\mathrm{L}^{2}(0,\tau)}\\
		&\leq C \bigg\{\big(\| \u_{0}\|_{\mathbb{H}^{2}(\mathcal{O})}\nonumber+\tau^{1/2}L+1\big)\|k_{1}-k_{2}\|_{\mathrm{L}^{2}(0,\tau)}\\&\quad+\tau^{1/2}\bigg[3L+\big(\|\u_{0}\|_{\mathbb{H}^{2}(\mathcal{O})}+\tau^{1/2}L\big)(\tau^{1/4}+1)\bigg]\|\tilde{\v}_{1}-\tilde{\v}_{2}\|_{\mathrm{H}^{1}(0,\tau;\mathbb{H}_{0}^{1}(\mathcal{O}) \cap \mathbb{H}^{2}(\mathcal{O}))}\bigg\}\\
		&\leq C\tau^{1/2}\bigg[\big(\| \u_{0}\|_{\mathbb{H}^{2}(\mathcal{O})}+\tau^{1/2}L+1\big) \alpha\bigg\{\bigg(\mu_{0} \tau^{1/2}  + \tau L\bigg) \| \Delta \varphi \|_{\mathbb{L}^{2}(\mathcal{O})}\\&\quad+2\bigg(\tau^{1/2}\|\u_{0}\|_{\mathbb{H}^{2}(\mathcal{O})}+\tau L+L\bigg) \| \nabla\varphi\|_{\mathbb{L}^{2}(\mathcal{O})} \bigg\} \\&\quad+ \bigg\{3L+\big(\|\u_{0}\|_{\mathbb{H}^{2}(\mathcal{O})}+\tau^{1/2}L\big)(\tau^{1/4}+1)\bigg\}\bigg]\|\tilde{\v}_{1}-\tilde{\v}_{2}\|_{\mathrm{H}^{1}(0,\tau;\mathbb{H}_{0}^{1}(\mathcal{O}) \cap \mathbb{H}^{2}(\mathcal{O}))}\nonumber\\&\quad+ C\alpha\tau^{1/2}L
		\bigg(\| \u_{0}\|_{\mathbb{H}^{2}(\mathcal{O})}+\tau^{1/2}L+1\bigg)\| \Delta \varphi \|_{\mathbb{L}^{2}(\mathcal{O})}\|\tilde{k}_{1}-\tilde{k}_{2}\|_{\mathrm{L}^{2}(0,\tau)}.
		\end{align*}
		If we take $\tau >0$ small enough, then we have $\Psi \colon \mathcal{V}(\tau,L) \to \mathcal{V}(\tau,L) $ is a contraction mapping. Thus, from the contraction mapping theorem, one can conclude that for a small time $\tau$, there exists a unique solution $(\v,k) \in \mathrm{H}^{1}(0,\tau;\mathbb{H}_{0}^{1}(\mathcal{O}) \cap \mathbb{H}^{2}(\mathcal{O})) \times \mathrm{L}^{2}(0,\tau)$ to the equivalent system \eqref{3.3}-\eqref{3.7} in $[0,\tau]$, and the proof is completed.
	\end{proof}
	\begin{remark}
		Due to the presence of convective (nonlinear) term $(\u \cdot \nabla)\u$ in the Kelvin-Voigt fluid equation \eqref{1.1}, we are not able  to prove the global existence and uniqueness of solutions. The authors in  \cite{CG} obtained global solvability results for bounded nonlinearities only (cf. Lemma 8.3, \cite{CG}). Thus, one cannot apply the same techniques used in \cite{CG} to obtain the global solvability results. This problem will be addressed in a future work. 
	\end{remark}
	\section{Oseen Type Equations}\label{sec5}\setcounter{equation}{0}
	In order to make use of the techniques available in \cite{CG} for the global solvability results, we consider an inverse problem for  Oseen type equations corresponding to Kelvin-Voigt fluids with memory in this section. Let us first consider the following the Oseen type equations:
	\begin{align}
	\nonumber \partial_{t}{\u}-\mu_{1} \partial_{t} \Delta \u-\mu_{0} \Delta \u +(\u_{\infty}\cdot \nabla)\u& -\int_{0}^{t}k(t-s)\Delta \u(s) \d s \\ +\ \nabla p&=\mathbf{0},  \ \  \mbox{ in } \ \mathcal{O} \times (0,T),\label{5.111}  \\
	\nabla \cdot \u&=0,  \ \ \mbox{ in } \  \mathcal{O} \times (0,T),\label{5.112} \\
	\u&=\mathbf{0},  \ \ \mbox{ on } \ \partial \mathcal{O} \times [0,T),\label{5.113}  \\
	\u&=\u_{0},  \ \mbox{ in } \ \mathcal{O} \times \{0\}\label{5.114},
	\end{align}
	where $\u_{\infty}$ is a divergence free vector field with $\u_{\infty} \in \mathbb{H}_{0}^{1}(\mathcal{O}) \cap \mathbb{H}^{2}(\mathcal{O})$. 
	\subsection{The inverse problem}\label{IP2}
	For $T>0$, the inverse problem is to determine $\tau \in (0,T]$,
	\begin{align}\label{5.1}
	\u \in \mathrm{H}^{2}(0,\tau;\mathbb{H}_{0}^{1}(\mathcal{O}) \cap \mathbb{H}^{2}(\mathcal{O}))\ \mbox{ and }\ k \in \mathrm{L}^{2}(0,\tau),
	\end{align}
	such that $(\u,k)$ satisfies the system:
	\begin{subequations}
		\begin{align}
		\nonumber \partial_{t}{\u}-\mu_{1} \partial_{t} \Delta \u-\mu_{0} \Delta \u +(\u_{\infty}\cdot \nabla)\u& -\int_{0}^{t}k(t-s)\Delta \u(s) \d s \\ +\ \nabla p&=\mathbf{0},  \ \  \mbox{ in } \ \mathcal{O} \times (0,\tau),\label{5.2}  \\
		\nabla \cdot \u&=0,  \ \ \mbox{ in } \ \mathcal{O} \times (0,\tau),\label{5.3} \\
		\u&=\mathbf{0},  \ \ \mbox{ on } \ \partial \mathcal{O} \times [0,\tau),\label{5.4}  \\
		\u&=\u_{0},  \ \mbox{ in } \  \mathcal{O} \times \{0\},\label{5.5} \\
		\int_{\mathcal{O}}(I-\mu_{1} \Delta)\varphi(x) \cdot \u(x,t) \d x&=r(t), \  \ \text { in } \  (0,\tau).\label{5.6}
		\end{align}
	\end{subequations}
	We solve the inverse problem under the following assumptions on the data: 
	\begin{itemize}
		\item [(H1)] $\u_{0} \in \mathbb{H}_{0}^{1}(\mathcal{O}) \cap \mathbb{H}^{2}(\mathcal{O}), \ \nabla \cdot \u_{0}=0, \ \ \mbox{in}  \ \mathcal{O}.$
		\item[(H2)] $\varphi \in \mathbb{H}_{0}^{1}(\mathcal{O}) \cap \mathbb{H}^{2}(\mathcal{O}), \  \nabla \cdot \varphi=0, \ \ \mbox{ in }  \ \mathcal{O}.$
		\item[(H3)] $\alpha^{-1}:= \int_{\mathcal{O}} \varphi \cdot \Delta \u_{0} \d x \ne 0.$
		\item[(H4)] $\v_{0}:=(I-\mu_{1} \Delta)^{-1}\big(\mu_{0} \Delta \u_{0}-(\u_{\infty} \cdot \nabla)\u_{0}- \nabla p_{0}\big)  \in \mathbb{H}_{0}^{1}(\mathcal{O}),$ $\nabla \cdot \v_{0}=0, 
		\ \ \mbox{ in }\ \mathcal{O}, \\  p_{0} \in \mathrm{H}^{1}(\mathcal{O}).$
		\item[(H5)] $r \in \mathrm{H}^{2}(0,T), $ with  \begin{align*} \int_{\mathcal{O}}(I-\mu_{1} \Delta)\varphi(x) \cdot \u_{0}(x) \d x&=r(0), \\ \int_{\mathcal{O}}\big(\mu_{0} \Delta \u_{0}(x)-(\u_{\infty}(x) \cdot \nabla)\u_{0}(x)\big) \cdot \varphi(x) \d x &=r'(0).\end{align*}
	\end{itemize}
	The following theorem provides an equivalent form of the system \eqref{5.2}-\eqref{5.6}.  
	\begin{theorem}\label{Equi-form1}
		Let the assumptions $\emph{(H1)-(H5)}$ hold. Let $(\u,k)$ be a solution of the system \eqref{5.2}-\eqref{5.6} up to $T$ such that
		\begin{align}\label{5.7}
		\u \in \mathrm{H}^{2}(0,T;\mathbb{H}_{0}^{1}(\mathcal{O}) \cap \mathbb{H}^{2}(\mathcal{O})) \ \mbox{ and } \ k \in \mathrm{L}^{2}(0,T).
		\end{align}
		Then $\v:=\partial_{t}{\u}$ and $k$ verify the conditions
		\begin{align}\label{5.8}
		\v \in \mathrm{H}^{1}(0,T;\mathbb{H}_{0}^{1}(\mathcal{O}) \cap \mathbb{H}^{2}(\mathcal{O}))\ \mbox{ and }\  k \in \mathrm{L}^{2}(0,T),
		\end{align}
		and solve the system
		\begin{subequations}
			\begin{align}
			\nonumber \partial_{t}{\v}-\mu_{1} \partial_{t} \Delta \v-\mu_{0} \Delta \v- k\Delta \u_{0} &-\int_{0}^{t}k(t-s)\Delta \v(s)\d s+\big(\u_{\infty}\cdot \nabla\big)\v\\ +\nabla \partial_{t}p&=\mathbf{0}, \ \  \mbox{ in }  \mathcal{O} \times (0,T), \label{5.9}\\
			\nabla \cdot \v&=0,  \  \  \mbox{ in } \ \mathcal{O} \times (0,T),\label{5.10} \\
			\v&=\mathbf{0}, \ \  \mbox{ on } \ \partial \mathcal{O} \times [0,T),\label{5.11} \\
			\v&=\v_{0}, \   \mbox{ in } \ \mathcal{O} \times \{0\},\label{5.12}
			\end{align}
		\end{subequations}
		with 
		\begin{equation}\label{5.13}\tag{5.9e}
		\begin{aligned}
		& k(t)=\alpha \bigg\{r''(t)-\mu_{0} \int_{\mathcal{O}}\v \cdot \Delta \varphi \d x-\int_{\mathcal{O}}((\u_{\infty}\cdot \nabla)\varphi\big)\cdot\v\d x -\int_{0}^{t}\int_{\mathcal{O}}k(t-s) \v \cdot\Delta \varphi \d x \d s\bigg\}.
		\end{aligned}
		\end{equation}
		On the other hand, under the setting $\u(t):=\u_{0}+\int_{0}^{t}\v(s) \d s$, if $(\v,k)$ satisfies \eqref{5.8} and is a solution to the system \eqref{5.9}-\eqref{5.13}, then considering to the above setting, $(\u,k)$  satisfies \eqref{5.7} and is a solution to the system \eqref{5.2}-\eqref{5.6}. 
	\end{theorem}
	\begin{proof}
		The proof constitutes the similar arguments as given in the proof of Theorem \ref{Equi-form}.
	\end{proof}
	The following theorem gives similar results for the system \eqref{5.2}-\eqref{5.6}, as we have proved in Theorem \ref{thm1} for the system \eqref{2.2}-\eqref{2.6}.
	\begin{theorem}[Local in time existence]\label{Thm1}
		Let the assumptions $(H1)$-$(H5)$ hold. Then there exists $\tau \in (0,T]$, such that the inverse problem \ref{IP2} has a unique solution $$(\u,k) \in \mathrm{H}^{2}(0,\tau;\mathbb{H}_{0}^{1}(\mathcal{O}) \cap \mathbb{H}^{2}(\mathcal{O})) \times \mathrm{L}^{2}(0,\tau).$$
	\end{theorem}
	\begin{proof}
		Let us define the space
		\begin{align*}
		\mathbf{V}(\tau,M):=&\bigg\{(\tilde{\v},\tilde{k}) \in  \mathrm{H}^{1}(0,\tau;\mathbb{H}_{0}^{1}(\mathcal{O}) \cap \mathbb{H}^{2}(\mathcal{O})) \times \mathrm{L}^{2}(0,\tau):
		\nabla \cdot \tilde{\v}=0, \ \mbox{ in } \ \mathcal{O} \times (0,\tau),\\
		&\quad\tilde{\v}=\mathbf{0}, \  \mbox{ on } \ \partial \mathcal{O} \times (0,\tau),\  
		\tilde{\v}=\v_{0}, \    \mbox{ in } \ \mathcal{O} \times \{0\}  \ \mbox{ and } \ \nonumber\\&\quad \|\tilde{\v}\|_{\mathrm{H}^{1}\left(0,\tau;\mathbb{H}_{0}^{1}(\mathcal{O}) \cap \mathbb{H}^{2}(\mathcal{O})\right)} +\|\tilde{k}\|_{\mathrm{L}^{2}(0,\tau)} \leq M  \bigg\},
		\end{align*}
		where $M$ is a positive constant, which will be determined later. We also define the mapping $ \Gamma \colon \mathbf{V}(\tau,M) \to \mathbf{V}(\tau,M)$ such that
		$(\tilde{\v},\tilde{k}) \mapsto (\v,k)$ 
		through \\
		\begin{align}\label{6a}
		&k(t):=\alpha \bigg\{r''-\mu_{0} \int_{\mathcal{O}}\tilde{\v} \cdot \Delta \varphi \d x -\int_{0}^{t}\int_{\mathcal{O}}\tilde{k}(t-s) \tilde{\v} \cdot\Delta \varphi \d x \d s-\int_{\mathcal{O}}\big((\u_{\infty}\cdot \nabla)\varphi\big) \cdot \tilde{\v}\d x \bigg\},
		\end{align}
		and the initial boundary value problem
		\begin{equation}\label{6b}
		\left\{
		\begin{aligned}
		\partial_{t}{\v}-\mu_{1} \partial_{t} \Delta \v-\mu_{0} \Delta \v- k\Delta \u_{0}+(\u_{\infty}\cdot \nabla) \cdot \tilde{\v}&-\int_{0}^{t}k(t-s)\Delta \tilde{\v}(s)\d s\\+\nabla \partial_{t}p&=\mathbf{0}, \ \  \mbox{ in } \ \mathcal{O} \times (0,\tau), \\
		\nabla \cdot \v&=0, \ \   \mbox{ in }  \ \mathcal{O} \times (0,\tau), \\
		\v&=\mathbf{0}, \ \  \mbox{ on }  \ \partial \mathcal{O} \times [0,\tau), \\
		\v&=\v_{0}, \  \mbox{ in } \  \mathcal{O} \times \{0\}.
		\end{aligned}
		\right.
		\end{equation}

		In order to complete the proof of Theorem \ref{Thm1}, we need to show that the mapping  $ \Gamma \colon \mathbf{V}(\tau,M) \to \mathbf{V}(\tau,M)$ is a contraction map. The proof constitutes the similar arguments as given in the proof of Theorem \ref{thm1}.
	\end{proof}
	Next, we prove the global uniqueness result for our inverse problem \ref{IP2}. Let $0< \tau <\tau' \leq (2\tau) \wedge T$ and let $(\hat{\v},\hat{k})$ be the solution of the system \eqref{5.9}-\eqref{5.13} in $[0,\tau]$  from Theorem \ref{Thm1} for small $\tau >0$. We will prove that  $(\hat{\v},\hat{k})$ is extensible to a solution  $(\v,k) \in \mathrm{H}^{1}(0,\tau';\mathbb{H}_{0}^{1}(\mathcal{O}) \cap \mathbb{H}^{2}(\mathcal{O})) \times \mathrm{L}^{2}(0,\tau')$ in $[0,\tau']$. We start with the following lemma which is very useful in the proof of Theorem \ref{Thm2}.
	\begin{lemma}\label{Lemma7a}
		Let the assumptions $(H1)$-$(H5)$ hold. Let $0< \tau <\tau' \leq (2\tau) \wedge T$ and let $(\u,k)$ be a solution of the system \eqref{5.2}-\eqref{5.6} in $[0,\tau']$, with the regularity 
		\begin{align*}
		\u \in \mathrm{H}^{2}(0,\tau';\mathbb{H}_{0}^{1}(\mathcal{O}) \cap \mathbb{H}^{2}(\mathcal{O})), \ \  k \in \mathrm{L}^{2}(0,\tau').
		\end{align*}
		We set
		\begin{itemize}
			\item [(F1)] $\v:=\partial_{t} \u,$
			\item [(F2)] $ \hat{\v}:= \v_{|[0,\tau]},$
			\item [(F3)] $ \v_{\tau}(t):=\v(\tau+t), \ \ \  t \in [0,\tau'-\tau],$
			\item [(F4)] $	 \hat{k}:= k_{|[0,\tau]},$
			\item [(F5)] $	  k_{\tau}(t):=k(\tau+t), \ \ \ t \in[0,\tau'-\tau],$ 
			\item [(F6)] $ h(t):=-\int_{t}^{\tau}\hat{k}(\tau+ t- s)\Delta \hat{\v}(s) \d s+ \nabla\partial_{t} p(\tau+t), \ \ \ t \in[0,\tau'-\tau],$
			\item [(F7)] $	\u_{\tau}:=\v(\tau),$
			\item [(F8)] $ r_{\tau}(t):=r(\tau+t), \ \ \  t \in[0,\tau'-\tau].$
		\end{itemize}
		Then we have the following obvious conditions:
		\begin{itemize}
			\item [(I)] $\v_{\tau} \in \mathrm{H}^{1}(0,\tau'-\tau;\mathbb{H}_{0}^{1}(\mathcal{O}) \cap \mathbb{H}^{2}(\mathcal{O}))$,
			\item[(II)] $k_{\tau} \in \mathrm{L}^{2}(0,\tau'-\tau)$,
			\item[(III)] $\u_{\tau} \in \mathbb{H}_{0}^{1}(\mathcal{O}) \cap \mathbb{H}^{2}(\mathcal{O})$,
			\item[(IV)] $r_{\tau} \in \mathrm{H}^{2}(0,\tau'-\tau)$,
			\item [(V)] $h \in \mathrm{L}^{2}(0,\tau'-\tau; \mathbb{L}^{2}(\mathcal{O}))$,
			\item [(VI)] $(\v_{\tau},k_{\tau})$ solves the system:
			\begin{equation}\label{7a}
			\left\{
			\begin{aligned}
			\partial_{t}{\v_{\tau}}-\mu_{1} \partial_{t} \Delta \v_{\tau}-\mu_{0} \Delta \v_{\tau}- k_{\tau}\Delta \u_{0} &-\int_{0}^{t}k_{\tau}(t-s)\Delta \hat{\v}(s)\d s\\-\int_{0}^{t}\hat{k}(t-s)\Delta \v_{\tau}(s)\d s+\big(\u_{\infty}\cdot \nabla\big)\v_{\tau}+h(t)&=\mathbf{0}, \ \ \mbox{ in } \ \mathcal{O} \times [0,\tau'-\tau], \\
			\nabla \cdot \v_{\tau}&=0, \ \  \mbox{ in } \ \mathcal{O} \times [0,\tau'-\tau], \\
			\v_{\tau}&=\mathbf{0}, \ \  \mbox{ on }\  \partial \mathcal{O} \times [0,\tau'-\tau], \\
			\v_{\tau}&=\u_{\tau}, \  \mbox{ in } \ \mathcal{O} \times \{0\},
			\end{aligned}
			\right.
			\end{equation}
			and
			\begin{equation}\label{7b}
			\left\{
			\begin{aligned}
			k_{\tau}(t)&=\alpha\bigg[r_{\tau}''(t)-\bigg\{\mu_{0}\int_{\mathcal{O}}  \v_{\tau} \cdot \Delta \varphi \d x +\int_{0}^{t}\int_{\mathcal{O}}k_{\tau}(t-s) \hat{\v}\cdot \Delta \varphi \d x\d s\\&\quad+\int_{0}^{t}\int_{\mathcal{O}}\hat{k}(t-s) \v_{\tau}\cdot \Delta \varphi \d x\d s +\int_{\mathcal{O}}\big((\u_{\infty}\cdot \nabla)\varphi\big)\cdot \v_\tau\d x\\&\quad+\int_{t}^{\tau}\int_{\mathcal{O}} \hat{k}(\tau+t-s)\hat{\v} \cdot \Delta\varphi\d x \d s\bigg\}\bigg], \ \ \ \ t\in[0,\tau'-\tau].
			\end{aligned}
			\right.
			\end{equation}
		\end{itemize}
		\begin{itemize}
			\item  [(VII)]
			On the other hand, if
			$(\hat{\v},\hat{k}) \in \mathrm{H}^{1}(0,\tau;\mathbb{H}_{0}^{1}(\mathcal{O}) \cap \mathbb{H}^{2}(\mathcal{O})) \times \mathrm{L}^{2}(0,\tau)$, solves the system \eqref{5.9}-\eqref{5.13} and $(\v_{\tau},k_{\tau}) \in \mathrm{H}^{1}(0,\tau'-\tau;\mathbb{H}_{0}^{1}(\mathcal{O}) \cap \mathbb{H}^{2}(\mathcal{O})) \times \mathrm{L}^{2}(0,\tau'-\tau)$ solves the system \eqref{7a}-\eqref{7b}. Set for $t \in [0,\tau']$,
			\begin{equation}\label{7c}
			(\v(t),k(t))=
			\left\{
			\begin{aligned}
			&(\hat{\v}(t),\hat{k}(t)),  &\text{ if } \ t \in [0,\tau],\\
			&(\v_{\tau}(t-\tau),k_{\tau}(t-\tau)), &\text{ if } \ t \in [\tau,\tau'].
			\end{aligned}
			\right.
			\end{equation}
			Then, $(\v,k) \in \mathrm{H}^{1}(0,\tau';\mathbb{H}_{0}^{1}(\mathcal{O}) \cap \mathbb{H}^{2}(\mathcal{O})) \times \mathrm{L}^{2}(0,\tau')$ and solves the system \eqref{5.9}-\eqref{5.13} in $[0,\tau']$.
		\end{itemize}
	\end{lemma}
	\begin{proof}
		The assertions		(I)-(VI) and (VIII) can be easily verified.
		Let us now prove (VII). First of all, from the system \eqref{5.9}-\eqref{5.13}, we have 
		\begin{equation}\label{7d}
		\left\{
		\begin{aligned}
		\partial_{t}{\v}(\tau+t)-\mu_{1} \partial_{t} \Delta \v(\tau+t)-\mu_{0} \Delta \v(\tau+t)&- k(\tau+t)\Delta \u_{0} -\int_{0}^{\tau+t}k(\tau+t-s)\Delta \v(s)\d s\\+\big(\u_\infty\cdot \nabla\big)\v(\tau+t)+\nabla \partial_{t}p(\tau+t)&=\mathbf{0}, \ \ \mbox{ in }\ \mathcal{O} \times [0,\tau'-\tau], \\
		\nabla \cdot \v(\tau+t)&=0, \ \  \mbox{ in }\ \mathcal{O} \times [0,\tau'-\tau],\\
		\v(\tau+t)&=\mathbf{0}, \ \  \mbox{ on }\ \partial \mathcal{O} \times [0,\tau'-\tau], \\
		\v(\tau)&=\u_{\tau}, \   \mbox{ in }\ \mathcal{O} \times \{0\},
		\end{aligned}
		\right.
		\end{equation}
		and
		\begin{equation}\label{7e}
		\left\{
		\begin{aligned}
		k(\tau+t)&=\alpha \bigg[r''(\tau+t)-\bigg\{\mu_{0} \int_{\mathcal{O}}\v(\tau+t) \cdot \Delta \varphi \d x+\int_{\mathcal{O}}(\u_\infty\cdot \nabla\big)\v(\tau+t)\d x  \\&\quad+\int_{0}^{\tau+t}\int_{\mathcal{O}}k(\tau+t-s) \v(\tau+t) \cdot\Delta \varphi \d x \d s, \ \ \text{ for all } \ t\in[0,\tau'-\tau].
		\end{aligned}
		\right.
		\end{equation}
		Next, as $\tau'-\tau \leq \tau$, if $t \in[0,\tau'-\tau]$, we have
		\begin{align}\label{7g}
		&\int_{0}^{\tau+t}k(\tau+t-s)\Delta \v(s)\d s\nonumber
		\\&=\int_{0}^{t}k(\tau+t-s)\Delta \v(s)\d s +\int_{t}^{\tau}k(\tau+t-s)\Delta \v(s)\d s+\int_{\tau}^{\tau+t}k(\tau+t-s)\Delta \v(s)\d s \nonumber\\
		&=\int_{0}^{t}k_{\tau}(t-s)\Delta \hat{\v}(s)\d s +\int_{0}^{t}\hat{k}(t-s)\Delta \v_{\tau}(s)\d s +\int_{t}^{\tau}\hat{k}(\tau+t-s)\Delta \hat{\v}(s)\d s. 
		\end{align}
		Substituting \eqref{7g} in \eqref{7d} and \eqref{7e}, we immediately get \eqref{7a} and \eqref{7b}.
	\end{proof}
	\begin{theorem}[Global in time uniqueness]\label{Thm2}
		Let the assumptions $(H1)$-$(H5)$ hold. Then, if $\tau \in (0,T]$, and the inverse problem  \ref{IP2} has two solutions $$(\u_{j},k_{j}) \in \mathrm{H}^{2}(0,\tau;\mathbb{H}_{0}^{1}(\mathcal{O}) \cap \mathbb{H}^{2}(\mathcal{O})) \times \mathrm{L}^{2}(0,\tau), \ j \in \{1,2\},$$  then $(\u_{1},k_{1})=(\u_{2},k_{2})$.
	\end{theorem}
	\begin{proof}
		Let $(\u_{j},k_{j}) \in \mathrm{H}^{2}(0,\tau;\mathbb{H}_{0}^{1}(\mathcal{O}) \cap \mathbb{H}^{2}(\mathcal{O})) \times \mathrm{L}^{2}(0,\tau), \ j \in \{1,2\}$, be solutions of the system \eqref{5.2}-\eqref{5.6}, for some $\tau \in (0,T].$	We claim that $\u_{1}=\u_{2}$ and $k_{1}=k_{2}$. We set $\v_{j}:=\partial_{t}\u_{j}, \ j \in \{1,2\}$. Then, from Theorem \ref{Equi-form1}, we have $(\v_{j},k_{j}) \in \mathrm{H}^{1}(0,\tau;\mathbb{H}_{0}^{1}(\mathcal{O}) \cap \mathbb{H}^{2}(\mathcal{O})) \times \mathrm{L}^{2}(0,\tau)$ and solves the system \eqref{5.9}-\eqref{5.13}. It suffices to show that  $\v_{1}=\v_{2}$ and $k_{1}=k_{2}$. Set $j \in\{1,2\}$, $\u_{j}(t):=\u_{0}+\int_{0}^{t} \v_{j}(s) \d s, $ (as in Theorem \ref{Equi-form1}).
		We need to prove
		\begin{align}\label{7h}
		\|\v_{1}-\v_{2}\|_{\mathrm{H}^{1}(0,\tau;\mathbb{H}_{0}^{1}(\mathcal{O}) \cap \mathbb{H}^{2}(\mathcal{O}))}+\|k_{1}-k_{2}\|_{ \mathrm{L}^{2}(0,\tau)}=0.
		\end{align}
		We have proved in Theorem \ref{Thm1} that for every $M \in \mathbb{R}^{+}$, there exists $\tau(M) \in (0,T]$, such that $\mbox{for all} \  \tau \in (0,\tau(M)],$ system \eqref{5.9}-\eqref{5.13} has a unique solution $$(\v,k) \in \mathrm{H}^{1}(0,\tau(M);\mathbb{H}_{0}^{1}(\mathcal{O}) \cap \mathbb{H}^{2}(\mathcal{O})) \times \mathrm{L}^{2}(0,\tau(M)),$$ such that
		$$\|\tilde{\v}\|_{\mathrm{H}^{1}(0,\tau(M);\mathbb{H}_{0}^{1}(\mathcal{O}) \cap \mathbb{H}^{2}(\mathcal{O}))} +\|\tilde{k}\|_{\mathrm{L}^{2}(0,\tau(M))} \leq M.$$
		This implies that, if $(\v_{j},k_{j}) \in \mathrm{H}^{1}(0,\tau;\mathbb{H}_{0}^{1}(\mathcal{O}) \cap \mathbb{H}^{2}(\mathcal{O})) \times \mathrm{L}^{2}(0,\tau), \  j \in \{1,2\} $, be solutions of the system \eqref{5.9}-\eqref{5.13}, then there exists $\tau_{1} \in (0,\tau]$, such that $(\v_{1},k_{1})$ and $(\v_{2},k_{2})$ coincide on $[0,\tau_{1}]$.
		In fact, we can set $$M_{1}:=\max\left\{\|\tilde{\v}_{j}\|_{\mathrm{H}^{1}(0,\tau;\mathbb{H}_{0}^{1}(\mathcal{O}) \cap \mathbb{H}^{2}(\mathcal{O}))} +\|\tilde{k}_{j}\|_{\mathrm{L}^{2}(0,\tau)}: j \in \{1,2\} \right\}$$ and obtain that $(\v_{1},k_{1})$ and $(\v_{2},k_{2})$ coincide in $[0,\tau(M_{1}) \wedge \tau]$.
		Let us define
		\begin{align}\label{7i}
		\tau_{1}:=\inf \left\{t \in (0,\tau]:\|\v_{1}-\v_{2}\|_{\mathrm{H}^{1}(0,t;\mathbb{H}_{0}^{1}(\mathcal{O}) \cap \mathbb{H}^{2}(\mathcal{O}))}+\|k_{1}-k_{2}\|_{ \mathrm{L}^{2}(0,t)}>0 \right\}.
		\end{align}
		If \eqref{7h} is not true, then it is obvious that $\tau_{1}$ is well defined and we have $\tau_{1} \in (0, \tau)$. Set, for $t \in [0,\tau - \tau_{1}]$, $j \in \{1,2\}$,
		\begin{align*}
		\v_{j}^{*}(t)=\v_{j}(\tau_{1}+t), \ \ \ \ k_{j}^{*}(t)=k_{j}(\tau_{1}+t).
		\end{align*}
		Keeping in mind that  $\v_{1}=\v_{2}$ and $k_{1}=k_{2}$ almost everywhere if $t \in [0,\tau_{1}]$.
		From conditions (I)-(VIII) in Lemma \ref{Lemma7a}, we obtain
		\begin{equation}\label{7j}
		\left\{
		\begin{aligned}
		\partial_{t}{(\v_{1}^{*}-\v_{2}^{*})}-\mu_{1} \partial_{t} \Delta (\v_{1}^{*}-\v_{2}^{*})&-\mu_{0} \Delta (\v_{1}^{*}-\v_{2}^{*})- (k_{1}^{*}-k_{2}^{*})\Delta \u_{0}\\-\int_{0}^{t}(k_{1}^{*}-k_{2}^{*})(t-s) \Delta \v_{1}(s)\d s&-\int_{0}^{t}k_{1}(t-s) \Delta (\v_{1}^{*}-\v_{2}^{*})(s)\d s\\+(\u_{\infty}\cdot \nabla)(\v_{1}^{*}-\v_{2}^{*}) 
		+\nabla\partial_{t}(p_{1}-p_{2})&=\mathbf{0},  \ \ \mbox{ in } \ \mathcal{O} \times  [0,\tau_{1} \wedge (\tau-\tau_{1})], \\
		\nabla \cdot (\v_{1}^{*}-\v_{2}^{*})&=0, \ \  \mbox{ in } \ \mathcal{O} \times  [0,\tau_{1} \wedge (\tau-\tau_{1})], \\
		(\v_{1}^{*}-\v_{2}^{*})&=\mathbf{0}, \ \ \mbox{ on } \ \partial \mathcal{O} \times [0,\tau_{1} \wedge (\tau-\tau_{1})], \\
		(\v_{1}^{*}-\v_{2}^{*})&=\mathbf{0},  \ \ \mbox{ in } \ \mathcal{O} \times \{0\} ,
		\end{aligned}
		\right.
		\end{equation}
		and
		\begin{equation}\label{7k}
		\left\{
		\begin{aligned}
		&(k_{1}^{*}-k_{2}^{*})(t)=-\alpha\bigg\{\mu_{0} \int_{\mathcal{O}} (\v_{1}^{*}-\v_{2}^{*})\cdot \Delta \varphi\d x+\int_{0}^{t}\int_{\mathcal{O}}(k_{1}^{*}-k_{2}^{*})(t-s)\v_{1}\cdot\Delta\varphi\d x\d s\\&\quad+\int_{0}^{t}\int_{\mathcal{O}}k_{1}(t-s) (\v_{1}^{*}-\v_{2}^{*})\cdot\Delta\varphi\d x\d s+\int_{\mathcal{O}}\big((\u_{\infty} \cdot \nabla)\varphi\big)\cdot(\v_{1}^*-\v_{2}^*)\d x.
		\end{aligned}
		\right.
		\end{equation}
		Assume that $\delta \in (0,\tau_{1} \wedge(\tau - \tau_{1})]$. Applying the same arguments to estimate $(k_{1}^{*}-k_{2}^{*})$ from equation \eqref{7k} as in Theorem \ref{thm1}, we obtain
		\begin{align*}
		\|k_{1}^{*}-k_{2}^{*}\|_{\mathrm{L}^{2}(0,\delta)} &\leq  \alpha\bigg\{\mu_{0}\|\v_{1}^{*}-\v_{2}^{*}\|_{\mathrm{L}^{2}(0,\delta; \mathbb{L}^{2}(\mathcal{O}))} \|\Delta \varphi\|_{\mathbb{L}^{2}(\mathcal{O})}\\&\quad+\delta^{1/2}\|k_{1}^{*}-k_{2}^{*}\|_{\mathrm{L}^{2}(0,\delta)}\|\Delta \varphi\|_{\mathbb{L}^{2}(\mathcal{O})} \|\v_{1}\|_{\mathrm{L}^{2}(0,\delta; \mathbb{L}^{2}(\mathcal{O}))} \\&\quad+\delta^{1/2}\|k_{1}\|_{\mathrm{L}^{2}(0,\delta)} \|\v_{1}^{*}-\v_{2}^{*}\|_{\mathrm{L}^{2}(0,\delta; \mathbb{L}^{2}(\mathcal{O}))} \|\Delta \varphi\|_{\mathbb{L}^{2}(\mathcal{O})}\\&\quad+C\|\u_{\infty}\|_{\mathbb{H}^{2}(\mathcal{O})}\|\nabla\varphi\|_{\mathbb{L}^{2}(\mathcal{O})}\|\v_{1}^{*}-\v_{2}^{*}\|_{\mathrm{L}^{2}(0,\delta; \mathbb{L}^{2}(\mathcal{O}))}\\&\leq C(\delta)\|\v_{1}^{*}-\v_{2}^{*}\|_{\mathrm{L}^{2}(0,\delta; \mathbb{L}^{2}(\mathcal{O}))}+C\delta^{1/2}\|k_{1}^{*}-k_{2}^{*}\|_{\mathrm{L}^{2}(0,\delta)}
		\end{align*}
		Choose $\delta>0$ such that $C\delta^{1/2}\leq\frac{1}{2}$, we obtain 
		\begin{align}\label{7l}
		\|k_{1}^{*}-k_{2}^{*}\|_{\mathrm{L}^{2}(0,\delta)} \leq C(\delta)\|\v_{1}^{*}-\v_{2}^{*}\|_{\mathrm{L}^{2}(0,\delta; \mathbb{L}^{2}(\mathcal{O}))}.
		\end{align}
		Taking the inner product with $(\v_{1}^{*}-\v_{2}^{*})(\cdot)$ to the first equation in \eqref{7j} and then using integration by parts, we derive that 
		\begin{align*}
		\frac{1}{2} &\frac{\d}{\d t}\left(\|(\v_{1}^{*}-\v_{2}^{*})(t)\|_{\L^2(\mathcal{O})}^2+\mu_{1}\|\nabla(\v_{1}^{*}-\v_{2}^{*})(t)\|_{\L^2(\mathcal{O})}^2\right)+\mu_0\|\nabla(\v_{1}^{*}-\v_{2}^{*})(t)\|_{\L^2(\mathcal{O})}^2\nonumber\\&= -((\u_{\infty}\cdot\nabla)(\v_{1}^{*}-\v_{2}^{*})(t),(\v_{1}^{*}-\v_{2}^{*})(t))-((k^*_{1}-k_{2}^*)*\nabla\v_{1}(t),\nabla(\v_{1}^{*}-\v_{2}^{*})(t))\\&\quad-((k_{1}^*-k_{2}^*)\nabla\u_0,\nabla(\v_{1}^{*}-\v_{2}^{*})(t))-(k_{1}*\nabla(\v_{1}^{*}-\v_{2}^{*})(t),\nabla(\v_{1}^{*}-\v_{2}^{*})(t))\nonumber\\&\leq \|\u_{\infty}\|_{\L^\infty(\mathcal{O})}\|\nabla(\v_{1}^{*}-\v_{2}^{*})(t)\|_{\L^2(\mathcal{O})}\|(\v_{1}^{*}-\v_{2}^{*})(t)\|_{\L^2(\mathcal{O})}\nonumber\\&\quad+\|(k^*_{1}-k_{2}^*)*\nabla\v_{1}(t)\|_{\L^2(\mathcal{O})}\|\nabla(\v_{1}^{*}-\v_{2}^{*})(t)\|_{\L^2(\mathcal{O})}\nonumber\\&\quad+|(k^*_{1}-k_{2}^*)(t)|\|\nabla\u_0\|_{\L^2(\mathcal{O})}\|\nabla(\v_{1}^{*}-\v_{2}^{*})(t)\|_{\L^2(\mathcal{O})}\nonumber\\&\quad+\|k_{1}*\nabla(\v_{1}^{*}-\v_{2}^{*})(t)\|_{\L^2(\mathcal{O})}\|\nabla(\v_{1}^{*}-\v_{2}^{*})(t)\|_{\L^2(\mathcal{O})}\nonumber\\&\leq C\|\u_{\infty}\|_{\H^2(\mathcal{O})}\|\nabla(\v_{1}^{*}-\v_{2}^{*})(t)\|^2_{\L^2(\mathcal{O})}+\frac{\mu_{0}}{2}
		\|\nabla(\v_{1}^{*}-\v_{2}^{*})(t)\|_{\L^2(\mathcal{O})}^2\\&\quad+\frac{3}{2\mu_{0}}\|(k^*_{1}-k_{2}^*)*\nabla\v_{1}(t)\|_{\L^2(\mathcal{O})}^2+\frac{3}{2\mu_{0}}|(k^*_{1}-k_{2}^*)(t)|^2\|\nabla\u_0\|_{\L^2(\mathcal{O})}^2\\&\quad+\frac{3}{2\mu_{0}}\|k_{1}*\nabla(\v_{1}^{*}-\v_{2}^{*})(t)\|_{\L^2(\mathcal{O})}^2,
		\end{align*}
		for a.e. $t\in[0,\delta]$. Integrating the above inequality from $0$ to $t$, we find 
		\begin{align*}
		&\|(\v_{1}^{*}-\v_{2}^{*})(t)\|_{\L^2(\mathcal{O})}^2+\mu_{1}\|\nabla(\v_{1}^{*}-\v_{2}^{*})(t)\|_{\L^2(\mathcal{O})}^2+\mu_0\int_{0}^{t}\|\nabla(\v_{1}^{*}-\v_{2}^{*})(s)\|_{\L^2(\mathcal{O})}^2 \d s\nonumber \\&\leq \mu_{1}\|\nabla(\v_{1}^{*}-\v_{2}^{*})(0)\|_{\L^2(\mathcal{O})}^2 +C\|\u_{\infty}\|_{\H^2(\mathcal{O})}\int_0^t\|\nabla(\v_{1}^{*}-\v_{2}^{*})(s)\|_{\L^2(\mathcal{O})}^2\d s\\&\quad+\frac{3}{\mu_{0}}t\|k^*_{1}-k_{2}^*\|_{\mathrm{L}^2(0,t)}^2\int_0^t\|\nabla\v_{1}(s)\|_{\L^2(\mathcal{O})}^2\d s+\frac{3}{\mu_{0}}\|k^*_{1}-k_{2}^*\|_{\mathrm{L}^2(0,t)}^2\|\nabla\u_0\|_{\L^2(\mathcal{O})}^2\\&\quad+\frac{3}{\mu_{0}}t\|k_{1}\|_{\mathrm{L}^2(0,t)}^2\int_0^t\|\nabla(\v_{1}^{*}-\v_{2}^{*})(s)\|_{\L^2(\mathcal{O})}^2\d s
		\nonumber \\&\leq C\|\nabla(\v_{1}^{*}-\v_{2}^{*})(0)\|_{\L^2(\mathcal{O})}^2+C(\delta)\|\v^*_{1}-\v_{2}^*\|_{\mathrm{L}^2(0,t;\L^2(\mathcal{O}))}^2\int_0^t\|\nabla\v_{1}(s)\|_{\L^2(\mathcal{O})}^2\d s\\&\quad+C\left(\|\u_{\infty}\|_{\H^2(\mathcal{O})}+t\|k_{1}\|_{\mathrm{L}^2(0,t)}^2\right)\int_0^t\|\nabla(\v_{1}^{*}-\v_{2}^{*})(s)\|_{\L^2(\mathcal{O})}^2\d s\nonumber\\&\quad+C(\delta)\|\v^*_{1}-\v_{2}^*\|_{\mathrm{L}^2(0,t;\L^2(\mathcal{O}))}^2\|\nabla\u_0\|_{\L^2(\mathcal{O})}^2,
		\end{align*}
		for all $t\in[0,\delta]$, where we used equation \eqref{7l}. An application of Gronwall's inequality gives 
		\begin{align*}
		\sup_{t\in[0,\delta]}	\|\nabla (\v_{1}^{*}-\v_{2}^{*})(t)\|_{\mathbb{L}^{2}(\mathcal{O})}\leq C(\|\nabla(\v_{1}^{*}-\v_{2}^{*})(0)\|_{\mathbb{L}^{2}(\mathcal{O})},\delta).
		\end{align*}
		Taking the inner product with $\partial_{t}(\v_{1}^{*}-\v_{2}^{*})(\cdot)$ to the first equation in  \eqref{7j}, we obtain
		\begin{align*}
		&\|\partial_{t}(\v_{1}^{*}-\v_{2}^{*})(t)\|_{\L^2(\mathcal{O})}^2+\mu_{1}\|\nabla\partial_{t}(\v_{1}^{*}-\v_{2}^{*})(t)\|_{\L^2(\mathcal{O})}^2+\frac{\mu_0}{2}\frac{\d}{\d t}\|\nabla(\v_{1}^{*}-\v_{2}^{*})(t)\|_{\L^2(\mathcal{O})}^2\nonumber\\&= -((\u_{\infty}\cdot\nabla)(\v_{1}^{*}-\v_{2}^{*})(t),\partial_{t}(\v_{1}^{*}-\v_{2}^{*})(t))-((k^*_{1}-k_{2}^*)*\nabla\v_{1}(t),\nabla\partial_{t}(\v_{1}^{*}-\v_{2}^{*})(t))\\&\quad+((k_{1}^*-k_{2}^*)\Delta\u_0,\partial_{t}(\v_{1}^{*}-\v_{2}^{*})(t))-(k_{1}*\nabla(\v_{1}^{*}-\v_{2}^{*})(t),\nabla\partial_{t}(\v_{1}^{*}-\v_{2}^{*})(t))\nonumber\\&\leq \|\u_{\infty}\|_{\L^\infty(\mathcal{O})}\|\nabla\partial_{t}(\v_{1}^{*}-\v_{2}^{*})(t)\|_{\L^2(\mathcal{O})}\|(\v_{1}^{*}-\v_{2}^{*})(t)\|_{\L^2(\mathcal{O})}+\|(k^*_{1}-k_{2}^*)*\nabla\v_{1}(t)\|_{\L^2(\mathcal{O})}\\&\quad\times\|\nabla\partial_{t}(\v_{1}^{*}-\v_{2}^{*})(t)\|_{\L^2(\mathcal{O})}+|(k^*_{1}-k_{2}^*)(t)|\|\Delta\u_0\|_{\L^2(\mathcal{O})}\|\partial_{t}(\v_{1}^{*}-\v_{2}^{*})(t)\|_{\L^2(\mathcal{O})}\\&\quad+\|k_{1}*\nabla(\v_{1}^{*}-\v_{2}^{*})(t)\|_{\L^2(\mathcal{O})}\|\nabla\partial_{t}(\v_{1}^{*}-\v_{2}^{*})(t)\|_{\L^2(\mathcal{O})}\nonumber\\&\leq C\|\u_{\infty}\|_{\H^2(\mathcal{O})}^2\|(\v_{1}^{*}-\v_{2}^{*})(t)\|_{\L^2(\mathcal{O})}^2+\frac{\mu_{1}}{2}
		\|\nabla\partial_{t}(\v_{1}^{*}-\v_{2}^{*})(t)\|_{\L^2(\mathcal{O})}^2\\&\quad+\frac{3}{2\mu_{1}}\|(k^*_{1}-k_{2}^*)*\nabla\v_{1}(t)\|_{\L^2(\mathcal{O})}^2+\frac{1}{2}|(k^*_{1}-k_{2}^*)(t)|^2\|\Delta\u_0\|_{\L^2(\mathcal{O})}^2\\&\quad+\frac{1}{2}\|\partial_{t}(\v_{1}^{*}-\v_{2}^{*})(t)\|_{\L^2(\mathcal{O})}^2+\frac{3}{2\mu_{1}}\|k_{1}*\nabla(\v_{1}^{*}-\v_{2}^{*})(t)\|_{\L^2(\mathcal{O})}^2,
		\end{align*}
		for a.e. $t \in[0,\delta]$. Integrating the above inequality from $0$ to $t$, we get
		\begin{align*}
		&\mu_0\|\nabla(\v_{1}^{*}-\v_{2}^{*})(t)\|_{\L^2(\mathcal{O})}^2+\int_{0}^{t}\|\partial_{t}(\v_{1}^{*}-\v_{2}^{*})(s)\|_{\L^2(\mathcal{O})}^2\d s+\mu_{1}\int_{0}^{t}\|\nabla\partial_{t}(\v_{1}^{*}-\v_{2}^{*})(s)\|_{\L^2(\mathcal{O})}^2\d s\\&\leq \mu_0\|\nabla(\v_{1}^{*}-\v_{2}^{*})(0)\|_{\L^2(\mathcal{O})}^2+C\|\u_{\infty}\|_{\H^2(\mathcal{O})}^2\int_0^t\|(\v_{1}^{*}-\v_{2}^{*})(s)\|_{\L^2(\mathcal{O})}^2\d s\\&\quad+\frac{3}{\mu_{1}}t\|k^*_{1}-k_{2}^*\|_{\L^2(0,t)}^2\int_0^t\|\nabla\v_{1}(s)\|_{\L^2(\mathcal{O})}^2\d s+\|k^*_{1}-k_{2}^*\|_{\L^2(0,t)}^2\|\Delta\u_0\|_{\L^2(\mathcal{O})}^2\\&\quad+\frac{3}{\mu_{1}}t\|k_{1}\|_{\L^2(0,t)}^2\int_0^t\|\nabla(\v_{1}^{*}-\v_{2}^{*})(s)\|_{\L^2(\mathcal{O})}^2\d s,
		\end{align*}
		for all $t \in [0,\delta]$. Thus, it is immediate from equation \eqref{7l} that $ \partial_{t}(\v_{1}^{*}-\v_{2}^{*}) \in \mathrm{L}^{2}(0,\delta;\mathbb{H}^{1}(\mathcal{O}))$.
		Taking divergence of the first equation in  \eqref{7j}, we deduce that 
		\begin{align*}
		\partial_{t}(p_1-p_2)=(-\Delta)^{-1}\big[\nabla \cdot[(\u_{\infty}\cdot \nabla)(\v_{1}^{*}-\v_{2}^{*})]\big],
		\end{align*}
		in the weak sense.	Therefore, we easily have 
		\begin{align*}
		\sup_{t\in[0,\delta]}\|\nabla \partial_{t}(p_1-p_2)(t)\|_{\mathbb{L}^2(\mathcal{O})}\nonumber&\leq C\sup_{t\in[0,\delta]}\|(\u_{\infty}\cdot \nabla)(\v_{1}^{*}-\v_{2}^{*})(t)\|_{\mathbb{L}^2(\mathcal{O})}\\&\leq C\sup_{t\in[0,\delta]}\left(\|\u_\infty\|_{\L^{\infty}(\mathcal{O})}\|\nabla(\v_{1}^{*}-\v_{2}^{*})(t)\|_{\L^2(\mathcal{O})}\right),
		\end{align*}
		which is finite and hence $ \nabla\partial_{t}(p_1-p_2)\in\mathrm{L}^{\infty}(0,\delta;\mathbb{L}^2(\mathcal{O}))$. Taking the inner product with $-\Delta(\v_{1}^{*}-\v_{2}^{*})(\cdot)$ to the first equation in  \eqref{7j}, we obtain
		\begin{align*}
		&\frac{1}{2} \frac{\d}{\d t}\left(\|\nabla(\v_{1}^{*}-\v_{2}^{*})(t)\|_{\L^2(\mathcal{O})}^2+\mu_{1}\|\Delta(\v_{1}^{*}-\v_{2}^{*})(t)\|_{\L^2(\mathcal{O})}^2\right)+\mu_0\|\Delta(\v_{1}^{*}-\v_{2}^{*})(t)\|_{\L^2(\mathcal{O})}^2\nonumber\\&\quad= ((\u_{\infty}\cdot\nabla)(\v_{1}^{*}-\v_{2}^{*})(t),\Delta(\v_{1}^{*}-\v_{2}^{*})(t))-((k^*_{1}-k_{2}^*)*\Delta\v_{1}(t),\Delta(\v_{1}^{*}-\v_{2}^{*})(t))\\&\quad-((k_{1}^*-k_{2}^*)\Delta\u_0,\Delta(\v_{1}^{*}-\v_{2}^{*})(t))-(k_{1}*\Delta(\v_{1}^{*}-\v_{2}^{*})(t),\Delta(\v_{1}^{*}-\v_{2}^{*})(t))\\&\quad+(\nabla\partial_{t}(p_1-p_2),\Delta(\v_{1}^{*}-\v_{2}^{*})(t))\nonumber\\&\leq \|\u_{\infty}\|_{\L^\infty(\mathcal{O})}\|\nabla(\v_{1}^{*}-\v_{2}^{*})(t)\|_{\L^2(\mathcal{O})}\|\Delta(\v_{1}^{*}-\v_{2}^{*})(t)\|_{\L^2(\mathcal{O})}+\|(k^*_{1}-k_{2}^*)*\Delta\v_{1}(t)\|_{\L^2(\mathcal{O})}\\&\quad\times\|\Delta(\v_{1}^{*}-\v_{2}^{*})(t)\|_{\L^2(\mathcal{O})}\nonumber+|(k^*_{1}-k_{2}^*)(t)|\|\Delta\u_0\|_{\L^2(\mathcal{O})}\|\Delta(\v_{1}^{*}-\v_{2}^{*})(t)\|_{\L^2(\mathcal{O})}\\&\quad+\|k_{1}*\Delta(\v_{1}^{*}-\v_{2}^{*})(t)\|_{\L^2(\mathcal{O})}\|\Delta(\v_{1}^{*}-\v_{2}^{*})(t)\|_{\L^2(\mathcal{O})}\nonumber\\&\quad+\|\nabla\partial_{t}(p_1-p_2)\|_{\mathbb{L}^2(\mathcal{O})}\|\Delta(\v_{1}^{*}-\v_{2}^{*})(t)\|_{\mathbb{L}^2(\mathcal{O})}\nonumber\\&\leq C\|\u_{\infty}\|_{\H^2(\mathcal{O})}^2\|\nabla(\v_{1}^{*}-\v_{2}^{*})(t)\|_{\L^2(\mathcal{O})}^2+\frac{\mu_{0}}{2}
		\|\Delta(\v_{1}^{*}-\v_{2}^{*})(t)\|_{\L^2(\mathcal{O})}^2\nonumber\\&\quad+\frac{5}{2\mu_{0}}\|(k^*_{1}-k_{2}^*)*\Delta\v_{1}(t)\|_{\L^2(\mathcal{O})}^2+\frac{5}{2\mu_{0}}|(k^*_{1}-k_{2}^*)(t)|^2\|\Delta\u_0\|_{\L^2(\mathcal{O})}^2\nonumber\\&\quad+\frac{5}{2\mu_{0}}\|k_{1}*\Delta(\v_{1}^{*}-\v_{2}^{*})(t)\|_{\L^2(\mathcal{O})}^2+\frac{5}{2\mu_{0}}\|\nabla\partial_{t}(p_1-p_2)\|_{\mathbb{L}^2(\mathcal{O})}^2,
		\end{align*}
		for a.e. $t\in[0,\delta]$. Integrating the above inequality from $0$ to $t$, we find 
		\begin{align*}
		&\|\nabla(\v_{1}^{*}-\v_{2}^{*})(t)\|_{\L^2(\mathcal{O})}^2+\mu_{1}\|\Delta(\v_{1}^{*}-\v_{2}^{*})(t)\|_{\L^2(\mathcal{O})}^2+\mu_0\int_{0}^{t}\|\Delta(\v_{1}^{*}-\v_{2}^{*})(s)\|_{\L^2(\mathcal{O})}^2 \d s\nonumber \\&\leq \|\nabla(\v_{1}^{*}-\v_{2}^{*})(0)\|_{\L^2(\mathcal{O})}^2+\mu_{1}\|\Delta(\v_{1}^{*}-\v_{2}^{*})(0)\|_{\L^2(\mathcal{O})}^2 +Ct\|\u_{\infty}\|_{\H^2(\mathcal{O})}^2\\&\quad+\frac{5}{\mu_{0}}t\|k^*_{1}-k_{2}^*\|_{\mathrm{L}^2(0,t)}^2\int_0^t\|\Delta\v_{1}(s)\|_{\L^2(\mathcal{O})}^2\d s+\frac{5}{\mu_{0}}\|k^*_{1}-k_{2}^*\|_{\mathrm{L}^2(0,t)}^2\|\Delta\u_0\|_{\L^2(\mathcal{O})}^2\nonumber\\&\quad+\frac{5}{\mu_{0}}t\|k_{1}\|_{\mathrm{L}^2(0,t)}^2\int_0^t\|\Delta(\v_{1}^{*}-\v_{2}^{*})(s)\|_{\L^2(\mathcal{O})}^2\d s +Ct	\sup_{t\in[0,\delta]}\|\nabla \partial_{t}(p_1-p_2)(t)\|_{\mathbb{L}^2(\mathcal{O})}
		\nonumber \\&\leq \|\nabla(\v_{1}^{*}-\v_{2}^{*})(0)\|_{\L^2(\mathcal{O})}^2+\mu_{1}\|\Delta(\v_{1}^{*}-\v_{2}^{*})(0)\|_{\L^2(\mathcal{O})}^2+Ct\|\u_{\infty}\|_{\H^2(\mathcal{O})}^2\\&\quad+ C(\delta)t\|\v^*_{1}-\v_{2}^*\|_{\mathrm{L}^2(0,t;\L^2(\mathcal{O}))}^2\int_0^t\|\Delta\v_{1}(s)\|_{\L^2(\mathcal{O})}^2\d s+C(\delta)\|\v^*_{1}-\v_{2}^*\|_{\mathrm{L}^2(0,t;\L^2(\mathcal{O}))}^2\|\Delta\u_0\|_{\L^2(\mathcal{O})}^2\nonumber\\&\quad+\frac{5}{\mu_{0}}t\|k_{1}\|_{\mathrm{L}^2(0,t)}^2\int_0^t\|\Delta(\v_{1}^{*}-\v_{2}^{*})(s)\|_{\L^2(\mathcal{O})}^2\d s+Ct	\sup_{t\in[0,\delta]}\|\nabla \partial_{t}(p_1-p_2)(t)\|_{\mathbb{L}^2(\mathcal{O})},
		\end{align*}
		for all $t\in[0,\delta]$, where we used equation \eqref{7l}. An application of Gronwall's inequality gives 
		\begin{align*}
		\sup_{t\in[0,\delta]}	\|\Delta (\v_{1}^{*}-\v_{2}^{*})(t)\|_{\mathbb{L}^{2}(\mathcal{O})}\leq C(\|(\v_{1}^{*}-\v_{2}^{*})(0)\|_{\mathbb{H}_{0}^{1}(\mathcal{O})},\|\Delta(\v_{1}^{*}-\v_{2}^{*})(0)\|_{\mathbb{L}^{2}(\mathcal{O})},\delta).
		\end{align*}
		Finally, taking the inner product with $-\Delta\partial_{t}(\v_{1}^{*}-\v_{2}^{*})(\cdot)$ to the first equation in  \eqref{7j}, we obtain
		\begin{align*}
		&\|\nabla\partial_{t}(\v_{1}^{*}-\v_{2}^{*})(t)\|_{\L^2(\mathcal{O})}^2+\mu_{1}\|\Delta\partial_{t}(\v_{1}^{*}-\v_{2}^{*})(t)\|_{\L^2(\mathcal{O})}^2+\frac{\mu_0}{2}\frac{\d}{\d t}\|\Delta(\v_{1}^{*}-\v_{2}^{*})(t)\|_{\L^2(\mathcal{O})}^2\nonumber\\&= ((\u_{\infty}\cdot\nabla)(\v_{1}^{*}-\v_{2}^{*})(t),\Delta\partial_{t}(\v_{1}^{*}-\v_{2}^{*})(t))-((k^*_{1}-k_{2}^*)*\Delta\v_{1}(t),\Delta\partial_{t}(\v_{1}^{*}-\v_{2}^{*})(t))\\&\quad-((k_{1}^*-k_{2}^*)\Delta\u_0,\Delta\partial_{t}(\v_{1}^{*}-\v_{2}^{*})(t))-(k_{1}*\Delta(\v_{1}^{*}-\v_{2}^{*})(t),\Delta\partial_{t}(\v_{1}^{*}-\v_{2}^{*})(t))\\&\quad+(\nabla\partial_t(p_1-p_2),\Delta\partial_{t}(\v_{1}^{*}-\v_{2}^{*})(t))\nonumber\\&\leq \|\u_{\infty}\|_{\L^\infty(\mathcal{O})}\|\Delta\partial_{t}(\v_{1}^{*}-\v_{2}^{*})(t)\|_{\L^2(\mathcal{O})}\|\nabla(\v_{1}^{*}-\v_{2}^{*})(t)\|_{\L^2(\mathcal{O})}\\&\quad+\|(k^*_{1}-k_{2}^*)*\Delta\v_{1}(t)\|_{\L^2(\mathcal{O})}\|\Delta\partial_{t}(\v_{1}^{*}-\v_{2}^{*})(t)\|_{\L^2(\mathcal{O})}\\&\quad+|(k^*_{1}-k_{2}^*)(t)|\|\Delta\u_0\|_{\L^2(\mathcal{O})}\|\Delta\partial_{t}(\v_{1}^{*}-\v_{2}^{*})(t)\|_{\L^2(\mathcal{O})}\\&\quad+\|k_{1}*\Delta(\v_{1}^{*}-\v_{2}^{*})(t)\|_{\L^2(\mathcal{O})}\|\Delta\partial_{t}(\v_{1}^{*}-\v_{2}^{*})(t)\|_{\L^2(\mathcal{O})}\nonumber\\&\quad+\|\nabla\partial_{t}(p_1-p_2)\|_{\mathbb{L}^2(\mathcal{O})}\|\Delta\partial_{t}(\v_{1}^{*}-\v_{2}^{*})(t)\|_{\mathbb{L}^2(\mathcal{O})}\nonumber\\&\leq C\|\u_{\infty}\|_{\H^2(\mathcal{O})}^2\|\nabla(\v_{1}^{*}-\v_{2}^{*})(t)\|_{\L^2(\mathcal{O})}^2+\frac{\mu_{1}}{2}
		\|\Delta\partial_{t}(\v_{1}^{*}-\v_{2}^{*})(t)\|_{\L^2(\mathcal{O})}^2\nonumber\\&\quad+\frac{5}{2\mu_{1}}\|(k^*_{1}-k_{2}^*)*\Delta\v_{1}(t)\|_{\L^2(\mathcal{O})}^2+\frac{5}{2\mu_1}|(k^*_{1}-k_{2}^*)(t)|^2\|\Delta\u_0\|_{\L^2(\mathcal{O})}^2\nonumber\\&\quad+\frac{5}{2\mu_{1}}\|k_{1}*\Delta(\v_{1}^{*}-\v_{2}^{*})(t)\|_{\L^2(\mathcal{O})}^2+\frac{5}{2\mu_{1}}\|\nabla\partial_{t}(p_1-p_2)\|_{\mathbb{L}^2(\mathcal{O})}^2,
		\end{align*}
		for a.e. $t\in[0,\delta]$. From the above inequality, we one can easily deduce that 
		\begin{align*}
		&\mu_{1}\|\Delta\partial_{t}(\v_{1}^{*}-\v_{2}^{*})(t)\|_{\L^2(\mathcal{O})}^2+\mu_{0}\frac{\d}{\d t}\|\Delta(\v_{1}^{*}-\v_{2}^{*})(t)\|_{\L^2(\mathcal{O})}^2\\&\leq C\|\u_{\infty}\|_{\H^2(\mathcal{O})}^2\|\nabla(\v_{1}^{*}-\v_{2}^{*})(t)\|_{\L^2(\mathcal{O})}^2+\frac{5}{\mu_{1}}\|(k^*_{1}-k_{2}^*)*\Delta\v_{1}(t)\|_{\L^2(\mathcal{O})}^2\\&\quad+\frac{5}{\mu_1}|(k^*_{1}-k_{2}^*)(t)|^2\|\Delta\u_0\|_{\L^2(\mathcal{O})}^2+\frac{5}{\mu_{1}}\|k_{1}*\Delta(\v_{1}^{*}-\v_{2}^{*})(t)\|_{\L^2(\mathcal{O})}^2\nonumber\\&\quad+\frac{5}{\mu_{1}}\|\nabla\partial_{t}(p_1-p_2)\|_{\mathbb{L}^2(\mathcal{O})}^2,
		\end{align*}
		for a.e. $t\in[0,\delta]$. Integrating the above inequality from $0$ to $t$, we find 
		\begin{align*}
		&\mu_{1}\int_0^t\|\Delta\partial_{t}(\v_{1}^{*}-\v_{2}^{*})(s)\|_{\L^2(\mathcal{O})}^2\d s+\mu_0\|\Delta(\v_{1}^{*}-\v_{2}^{*})(t)\|_{\L^2(\mathcal{O})}^2\\&\leq\mu_0\|\Delta(\v_{1}^{*}-\v_{2}^{*})(0)\|_{\L^2(\mathcal{O})}^2+ C\|\u_{\infty}\|_{\H^2(\mathcal{O})}^2\int_0^t\|\nabla(\v_{1}^{*}-\v_{2}^{*})(s)\|_{\L^2(\mathcal{O})}^2\d s\\&\quad+\frac{5}{\mu_{1}}t\|k^*_{1}-k_{2}^*\|_{\L^2(0,t)}^2\int_0^t\|\Delta\v_{1}(s)\|_{\L^2(\mathcal{O})}^2\d s+\frac{5}{\mu_1}\|k^*_{1}-k_{2}^*\|_{\L^2(0,t)}^2\|\Delta\u_0\|_{\L^2(\mathcal{O})}^2\\&\quad+\frac{5}{\mu_{1}}t\|k_{1}\|_{\L^2(0,t)}^2\int_0^t\|\Delta(\v_{1}^{*}-\v_{2}^{*})(s)\|_{\L^2(\mathcal{O})}^2\d s\nonumber+Ct	\sup_{t\in[0,\delta]}\|\nabla \partial_{t}(p_1-p_2)(t)\|_{\mathbb{L}^2(\mathcal{O})}.
		\end{align*}
		Thus, it is immediate that for $\delta \in (0,\tau_{1} \wedge(\tau - \tau_{1})]$
		\begin{align*}
		\|\v_{1}^{*}-\v_{2}^{*}\|_{\mathrm{H}^{1}(0,\delta;\mathbb{H}_{0}^{1}(\mathcal{O}) \cap \mathbb{H}^{2}(\mathcal{O}))} &\leq C(\delta)\|k_{1}^{*}-k_{2}^{*}\|_{\mathrm{L}^{2}(0,\delta)}.
		\end{align*}
		Using \eqref{7l} in the above estimate, we obtain
		\begin{align*}
		\|\v_{1}^{*}-\v_{2}^{*}\|_{\mathrm{H}^{1}(0,\delta;\mathbb{H}_{0}^{1}(\mathcal{O}) \cap \mathbb{H}^{2}(\mathcal{O}))} \leq C(\delta)\|\v_{1}^{*}-\v_{2}^{*}\|_{\mathrm{H}^{1}(0,\delta;\mathbb{H}_{0}^{1}(\mathcal{O}) \cap \mathbb{H}^{2}(\mathcal{O}))}.
		\end{align*}
		If $\delta$ is sufficiently small such that
		$\|\v_{1}^{*}-\v_{2}^{*}\|_{\mathrm{H}^{1}(0,\delta;\mathbb{H}_{0}^{1}(\mathcal{O}) \cap \mathbb{H}^{2}(\mathcal{O}))}=0$. Then, \eqref{7l} ensures that $k_{1}^{*}-k_{2}^{*}$ also vanish in some neighborhood of $0$, 
		which contradicts the definition of $\tau_{1}$. Hence, we conclude that $ \v_{1}=\v_{2}$ and $k_{1}=k_{2}$, which completes the proof.
	\end{proof}
	\begin{theorem}[Global in time existence and uniqueness]\label{Thm3}
		Let the assumptions $(H1)$-$(H5)$ hold. Let $T>0$. Then the inverse problem \ref{IP2} has a unique solution $$(\u,k) \in \mathrm{H}^{2}(0,T;\mathbb{H}_{0}^{1}(\mathcal{O}) \cap \mathbb{H}^{2}(\mathcal{O})) \times \mathrm{L}^{2}(0,T).$$
	\end{theorem}
	To prove Theorem \ref{Thm3}, that is, the global solvability of the inverse problem \eqref{5.2}-\eqref{5.6}, we prove the result for its equivalent form, that is, for any given time $T>0$, there exists a solution to the system \eqref{5.9}-\eqref{5.13}. We start with the following Lemmas, which will be useful in the sequel.
	\begin{lemma}\label{Lemma 8a}
		Let the	assumptions $(H1)$-$(H5)$ hold. Let $\tau\in(0,T)$ and let $$(\hat{\v},\hat{k})\in \mathrm{H}^{1}(0,\tau;\mathbb{H}_{0}^{1}(\mathcal{O}) \cap \mathbb{H}^{2}(\mathcal{O})) \times \mathrm{L}^{2}(0,\tau)$$
		solves the system \eqref{5.9}-\eqref{5.13} on $[0,\tau]$. Then there exists $\delta\in (0,T-\tau]$, such that  the solution $(\hat{\v},\hat{k})$ can be extended to a solution $(\v,k)$ on $[0,\tau + \delta]$, such that $$(\v,k) \in \mathrm{H}^{1}(0,\tau +\delta; \mathbb{H}_{0}^{1}(\mathcal{O})\cap \mathbb{H}^{2}(\mathcal{O}))\times \mathrm{L}^{2}(0,\tau +\delta).$$
	\end{lemma}
	\begin{proof}
		From the conditions (I)-(VIII), it suffices to show that the system \eqref{7a}-\eqref{7b}, with $\ h,\  \u_{\tau}$ be as defined in Lemma \ref{Lemma7a}, has a solution $$(\v_{\tau},k_{\tau}) \in \mathrm{H}^{1}(0,\delta;\mathbb{H}_{0}^{1}(\mathcal{O}) \cap \mathbb{H}^{2}(\mathcal{O})) \times \mathrm{L}^{2}(0,\delta),$$ for some $\delta\in (0,T-\tau]$. We define a space
		\begin{align*}
		\mathbf{V}(\delta,M):=&\bigg\{(\tilde{\v},\tilde{k}) \in  \mathrm{H}^{1}(0,\delta;\mathbb{H}_{0}^{1}(\mathcal{O}) \cap \mathbb{H}^{2}(\mathcal{O})) \times \mathrm{L}^{2}(0,\delta):
		\nabla \cdot \tilde{\v}=0, \ \  \mbox{ in } \ \mathcal{O} \times (0,\delta),\\
		&\quad\tilde {\v}=\mathbf{0}, \  \mbox{ on } \ \partial \mathcal{O} \times [0,\delta),\  
		\tilde{\v}=\u_{\tau}, \    \mbox{ in } \ \mathcal{O} \times \{0\} \ \ 
		\mbox{and} \nonumber\\&\quad  \|\tilde{\v}\|_{\mathrm{H}^{1}(0,\delta;\mathbb{H}_{0}^{1}(\mathcal{O}) \cap \mathbb{H}^{2}(\mathcal{O}))} +\|\tilde{k}\|_{\mathrm{L}^{2}(0,\delta)} \leq M  \bigg\},
		\end{align*}
		where $M$ is a positive constant, which will be determined later. We also define the mapping $ \Gamma \colon \mathbf{V}(\delta,M) \to \mathbf{V}(\delta,M)$ such that
		$(\tilde{\v},\tilde{k}) \mapsto (\v,k)$ 
		through \\
		\begin{equation}\label{8.1}
		\left\{
		\begin{aligned}
		\partial_{t}{\v_{\tau}}-\mu_{1} \partial_{t} \Delta \v_{\tau}-\mu_{0} \Delta \v_{\tau}- k_{\tau}\Delta \u_{0}& -\int_{0}^{t}k_{\tau}(t-s)\Delta \hat{\v}(s)\d s-\int_{0}^{t}\hat{k}(t-s)\Delta \tilde{\v}(s)\d s\\+\big(\u_{\infty}\cdot \nabla\big)\tilde{\v}+h(t)&=\mathbf{0}, \ \ \mbox{ in } \ \mathcal{O} \times [0,\delta], \\
		\nabla \cdot \v_{\tau}&=0, \ \  \mbox{ in } \ \mathcal{O} \times [0,\delta], \\
		\v_{\tau}&=\mathbf{0}, \ \  \mbox{ on } \ \partial \mathcal{O} \times [0,\delta], \\
		\v_{\tau}&=\u_{\tau}, \  \mbox{ in } \ \mathcal{O} \times \{0\},
		\end{aligned}
		\right.
		\end{equation}
		and
		\begin{equation}\label{8.2}
		\left\{
		\begin{aligned}
		k_{\tau}(t)&=\alpha\bigg[r_{\tau}''(t)-\bigg\{\mu_{0}\int_{\mathcal{O}} \tilde{\v}\cdot\Delta\varphi\d x +\int_{0}^{t}\int_{\mathcal{O}}\tilde{k}(t-s)\hat{\v}\cdot\Delta\varphi\d x\d s\\&\quad+\int_{0}^{t}\int_{\mathcal{O}}\hat{k}(t-s) \tilde{\v}\cdot\Delta\varphi\d x\d s+\int_{\mathcal{O}}\big((\u_{\infty}\cdot \nabla)\varphi\big)\cdot\tilde{\v}\d x\\&\quad+\int_{t}^{\tau}\int_{\mathcal{O}} \hat{k}(\tau+t-s)\hat{\v} \cdot \Delta\varphi\d x \d s\bigg\}\bigg], \ \ \ t\in [0,\delta].
		\end{aligned}
		\right.
		\end{equation}
		We just need to show that the map  $ \Gamma \colon \mathbf{V}(\delta,M) \to \mathbf{V}(\delta,M)$ is a contraction map.
		The proof constitutes the similar arguments as given in the proof of Theorem \ref{thm1}.
	\end{proof}
	\begin{lemma}\label{Lemma 8b}
		Let $T \in \mathbb{R}^{+}, \ \tau \in (0,T].$ Let  $\beta_{1},\beta_{2} \in \mathbb{R}^{+}$ and $\v \in \mathrm{H}^{1}(0,\tau;\mathbb{H}_{0}^{1}(\mathcal{O}) \cap \mathbb{H}^{2}(\mathcal{O}))$ be such that, $\mbox{for all} \  t \in [0,\tau],$
		\begin{align}\label{8.3}
		\|\v\|_{\mathrm{H}^{1}(0,t;\mathbb{H}_{0}^{1}(\mathcal{O}) \cap \mathbb{H}^{2}(\mathcal{O}))}^{2} \leq \beta_{1}+ \beta_{2}\|\v\|_{\mathrm{L}^{2}(0,t;\mathbb{H}_{0}^{1}(\mathcal{O}))}^{2}.
		\end{align}
		Then, 
		\begin{align}\label{8.4}
		\|\v\|_{\mathrm{H}^{1}(0,t;\mathbb{H}_{0}^{1}(\mathcal{O}) \cap \mathbb{H}^{2}(\mathcal{O}))} \leq C\big( \beta_{1}, \beta_{2}, \|\v(0)\|_{\mathbb{L}^{2}(\mathcal{O})}, T\big),
		\end{align}
		where $ C\big( \beta_{1}, \beta_{2}, \|\v(0)\|_{\mathbb{L}^{2}(\mathcal{O})}, T\big) \in \mathbb{R}^{+}. $
	\end{lemma}
	\begin{proof}
		Applying Gagliardo-Nirenberg's and Young's inequalities, we obtain, $\mbox{for all} \  \varepsilon \in \mathbb{R}^{+}$,
		\begin{align}\label{8.5}
		\|\v\|_{\mathrm{L}^{2}(0,t;\mathbb{H}_{0}^{1}(\mathcal{O}))}^{2} \nonumber&\leq C \|\v\|_{\mathrm{L}^{2}(0,t;\mathbb{L}^{2}(\mathcal{O}))}\|\v\|_{\mathrm{L}^{2}(0,t;\mathbb{H}_{0}^{1} (\mathcal{O})\cap \mathbb{H}^{2}(\mathcal{O}))} \\
		&\leq C\big(\varepsilon\|\v\|_{\mathrm{L}^{2}(0,t;\mathbb{H}_{0}^{1} (\mathcal{O})\cap \mathbb{H}^{2}(\mathcal{O}))}^{2}+C(\varepsilon)\|\v\|_{\mathrm{L}^{2}(0,t;\mathbb{L}^{2}(\mathcal{O}))}^{2}\big).
		\end{align}
		Substituting \eqref{8.5} in \eqref{8.3}, we have
		\begin{align}\label{8.6}
		\|\v\|_{\mathrm{H}^{1}(0,t;\mathbb{H}_{0}^{1}(\mathcal{O}) \cap \mathbb{H}^{2}(\mathcal{O}))}^{2} \leq \beta_{1}+\beta_{2}C \varepsilon \|\v\|_{\mathrm{L}^{2}(0,t;\mathbb{H}_{0}^{1} (\mathcal{O})\cap \mathbb{H}^{2}(\mathcal{O}))}^{2}+\beta_{2}C(\varepsilon)\|\v\|_{\mathrm{L}^{2}(0,t;\mathbb{L}^{2}(\mathcal{O}))}^{2},
		\end{align}
		for every $\varepsilon \in \mathbb{R}^{+}$. Choosing $\varepsilon$ in such a way that $\beta_{2}C\varepsilon \leq 1/2$, we get
		\begin{align}\label{8.7}
		\|\v\|_{\mathrm{H}^{1}(0,t;\mathbb{H}_{0}^{1}(\mathcal{O}) \cap \mathbb{H}^{2}(\mathcal{O}))}^{2} \leq 2\beta_{1}+C(\beta_{2})\|\v\|_{\mathrm{L}^{2}(0,t;\mathbb{L}^{2}(\mathcal{O}))}^{2}.
		\end{align}
		Since 	$\v \in \mathrm{H}^{1}(0,\tau;\mathbb{H}_{0}^{1}(\mathcal{O}) \cap \mathbb{H}^{2}(\mathcal{O}))$, $\v$ is absolutely continuous (see Sec. 5.9, Theorem 2, \cite{LCE}). Hence, $\mbox{for all} \ s \in (0,\tau]$,
		\begin{align*}
		\v(s)&=\v(0)+\int_{0}^{s}\partial_{t}\v(s_{1})\d s_{1},\\
		\|\v(s)\|_{\mathbb{L}^{2}(\mathcal{O})}^{2}&\leq  2\|\v(0)\|_{\mathbb{L}^{2}(\mathcal{O})}^{2}+2\left\|\int_{0}^{s}\partial_{t}\v(s_{1})\d s_{1}\right\|_{\mathbb{L}^{2}(\mathcal{O})}^{2}.
		\end{align*}
		Integrating the above estimate from $0$ to $t$, we obtain
		\begin{align}\label{8.8}
		\|\v\|_{\mathrm{L}^{2}(0,t;\mathbb{L}^{2}(\mathcal{O}))}^{2} \nonumber&\leq 2 t\|\v(0)\|^{2}_{\mathbb{L}^{2}(\mathcal{O})}+2\int_{0}^{t}\left\|\int_{0}^{s}\partial_{t}\v(s_{1})\d s_{1}\right\|_{\mathbb{L}^{2}(\mathcal{O})}^{2} \d s \\
		&\leq 2 t\|\v(0)\|^{2}_{\mathbb{L}^{2}(\mathcal{O})}+2\int_{0}^{t}s\|\partial_{t}\v\|_{\mathrm{L}^{2}(0,s;\mathbb{L}^{2}(\mathcal{O}))}^{2} \d s. 
		\end{align}
		If we set $z(t):=	\|\partial_{t}\v\|_{\mathrm{L}^{2}(0,t;\mathbb{L}^{2}(\mathcal{O}))}^{2}$, from \eqref{8.7}, we get, $ \mbox{for all} \  t \in [0,\tau]$,
		\begin{align}\label{8.9}
		z(t) \leq 2\beta_{1}+
		C(\beta_{2})t\|\v(0)\|^{2}_{\mathbb{L}^{2}(\mathcal{O})}+C(\beta_{2})\int_{0}^{t} sz(s)\d s .
		\end{align}
		An application of Gronwall's inequality in \eqref{8.9} yields
		\begin{align}\label{8.10}
		z(t) \leq C\big(\beta_{1}, \beta_{2}, \|\v(0)\|_{\mathbb{L}^{2}(\mathcal{O})}, T\big). 
		\end{align}
		Using \eqref{8.10} in \eqref{8.8}, $ \mbox{for all} \ t \in (0,\tau]$, we have
		\begin{align}\label{8.11}
		\|\v\|_{\mathrm{L}^{2}(0,t;\mathbb{L}^{2}(\mathcal{O}))}^{2}
		&\leq 2t\|\v(0)\|_{\mathbb{L}^{2}(\mathcal{O})}^{2}\nonumber+2\int_{0}^{t} s C\big(\beta_{1}, \beta_{2}, \|\v(0)\|_{\mathbb{L}^{2}(\mathcal{O})}, T\big)\d s\\
		&\leq 2t\|\v(0)\|_{\mathbb{L}^{2}(\mathcal{O})}^{2}+t^{2} C\big(\beta_{1}, \beta_{2}, \|\v(0)\|_{\mathbb{L}^{2}(\mathcal{O})}, T\big). 
		\end{align}
		From \eqref{8.7} and \eqref{8.11}, we conclude that
		\begin{align*}
		\|\v\|_{\mathrm{H}^{1}(0,t;\mathbb{H}_{0}^{1}(\mathcal{O}) \cap \mathbb{H}^{2}(\mathcal{O}))} \leq C\big(\beta_{1}, \beta_{2}, \|\v(0)\|_{\mathbb{L}^{2}(\mathcal{O})}, T\big),
		\end{align*}
		which completes the proof.
	\end{proof}
	Let us now provide an a-priori estimate on $(\v,k)$ defined by \eqref{7c}. Such an estimate is of crucial importance in the proof of Theorem \ref{Thm3}.
	\begin{lemma}\label{Lemma 8c}
		Let the	assumptions $(H1)$-$(H5)$ hold. Let $$(\hat{\v},\hat{k})\in \mathrm{H}^{1}(0,\tau;\mathbb{H}_{0}^{1}(\mathcal{O}) \cap \mathbb{H}^{2}(\mathcal{O})) \times \mathrm{L}^{2}(0,\tau)$$
		be a solution of the system \eqref{5.9}-\eqref{5.13} in $[0,\tau]$ with $0<\tau<T$. If $$(\v,k)\in \mathrm{H}^{1}(0,\tau +\delta; \mathbb{H}_{0}^{1}(\mathcal{O})\cap \mathbb{H}^{2}(\mathcal{O}))\times \mathrm{L}^{2}(0,\tau +\delta)$$ is a solution of the system \eqref{5.9}-\eqref{5.13} in $[0,\tau+\delta]$. Then, we can bound  $(\v,k)$ and there exists $C>0$, such that, $\mbox{ for all }  \delta \in (0,\tau\land (T-\tau)]$,
		$$\|\v\|_{\mathrm{H}^{1}(0,\tau+\delta;\mathbb{H}_{0}^{1}(\mathcal{O})\cap \mathbb{H}^{2}(\mathcal{O}))}+\|k\|_{\mathrm{L}^{2}(0,\tau+\delta)}\leq C.$$
	\end{lemma}
	\begin{proof}
		Owing to Lemma \ref{Lemma7a} and Theorem \ref{Thm2}, if we set
		\begin{align*}
		\v_{\tau}(t):=\v(\tau+t), \ \ k_{\tau}(t):=k(\tau+t), \ t \in [0,\delta],
		\end{align*}
		we get $(\v_{\tau},k_{\tau})$ solves the system \eqref{7a}-\eqref{7b}. Let $k=\hat{k}$.
		Applying the same arguments to estimate $\v_{\tau}$ from the system \eqref{7a} as in Theorem \ref{Thm2}, we obtain the following estimate
		\begin{align}\label{8.13}
		\|\v_{\tau}\|^2_{\mathrm{H}^{1}(0,t;\mathbb{H}_{0}^{1}(\mathcal{O}) \cap \mathbb{H}^{2}(\mathcal{O}))} \leq C\big(\u_{0},\hat{k},\hat{\v},T\big)\big[\|\u_{\tau}\|^2_{\mathbb{H}_{0}^{1}(\mathcal{O}) \cap \mathbb{H}^{2}(\mathcal{O})}+\|k_{\tau}\|^2_{ \mathrm{L}^{2}(0,t)}\big],
		\end{align}
		$\mbox{ for all} \ t \in [0,\delta]$. Similarly, we estimate $k_{\tau}$ from the equation \eqref{7b}, we find
		\begin{align}\label{8.14}
		\|k_{\tau}\|_{ \mathrm{L}^{2}(0,t)}^2 &\leq C(\hat{\v},\hat{k},T)\big[\|r_{\tau}''\|_{ \mathrm{L}^{2}(0,t)}^2 +	\|\v_{\tau}\|_{\mathrm{L}^{2}(0,t;\mathbb{H}_{0}^{1}(\mathcal{O}))}^2\big]\nonumber\\&\leq C(\hat{\v},\hat{k},T)\big[\|r_{\tau}''\|_{ \mathrm{L}^{2}(0,t)}^2 +\varepsilon\|\v_{\tau}\|_{\mathrm{L}^{2}(0,t;\mathbb{H}_{0}^{1}(\mathcal{O})\cap\mathbb{H}^2(\mathcal{O}))}^2+	C(\varepsilon)\|\v_{\tau}\|_{\mathrm{L}^{2}(0,t;\mathbb{H}_{0}^{1}(\mathcal{O}))}^2\big].
		\end{align}
		Substituting \eqref{8.14} in \eqref{8.13}, $\mbox{ for all} \ t \in [0,\delta]$, we obtain
		\begin{align}\label{8.15}
		\|\v_{\tau}\|^2_{\mathrm{H}^{1}(0,t;\mathbb{H}_{0}^{1}(\mathcal{O}) \cap \mathbb{H}^{2}(\mathcal{O}))} \nonumber&\leq C\big(\u_{0},\hat{k},\hat{\v},T\big)\big[\|\u_{\tau}\|^2_{\mathbb{H}_{0}^{1}(\mathcal{O}) \cap \mathbb{H}^{2}(\mathcal{O})}+\|r_{\tau}''\|^2_{ \mathrm{L}^{2}(0,t)}\\&\quad+\varepsilon\|\v_{\tau}\|_{\mathrm{L}^{2}(0,t;\mathbb{H}_{0}^{1}(\mathcal{O})\cap\mathbb{H}^2(\mathcal{O}))}^2+	C(\varepsilon)\|\v_{\tau}\|_{\mathrm{L}^{2}(0,t;\mathbb{H}_{0}^{1}(\mathcal{O}))}^2\big].
		\end{align}
		Choosing $\varepsilon$ to be sufficiently small, so that $\varepsilon C\big(\u_{0},\hat{k},\hat{\v},T\big)\leq 1/2,$ from \eqref{8.15}, we obtain
		\begin{align}\label{8.15*}
		\|\v_{\tau}\|^2_{\mathrm{H}^{1}(0,t;\mathbb{H}_{0}^{1}(\mathcal{O}) \cap \mathbb{H}^{2}(\mathcal{O}))} &\leq C\big(\u_{0},\hat{k},\hat{\v},T\big)\big[\|\u_{\tau}\|^2_{\mathbb{H}_{0}^{1}(\mathcal{O}) \cap \mathbb{H}^{2}(\mathcal{O})}+\|r_{\tau}''\|^2_{ \mathrm{L}^{2}(0,t)}+\|\v_{\tau}\|_{\mathrm{L}^{2}(0,t;\mathbb{H}_{0}^{1}(\mathcal{O}))}^{2}\big].
		\end{align}
		Thus, we conclude from Lemma \ref{Lemma 8b}, \eqref{8.15*} and \eqref{8.14} that
		$$\|\v\|_{\mathrm{H}^{1}(0,\tau+\delta;\mathbb{H}_{0}^{1}(\mathcal{O})\cap \mathbb{H}^{2}(\mathcal{O}))}+\|k\|_{\mathrm{L}^{2}(0,\tau+\delta)}\leq C,$$ which completes the proof.
	\end{proof}
	Let us now prove Theorem \ref{Thm3}.
	\begin{proof}[Proof of Theorem \ref{Thm3}]
		Global in time uniqueness follows from Theorem \ref{Thm2}. It remains to prove existence. For this, having Theorem \ref{Equi-form1}, it suffices to show the existence of solution to the system \eqref{5.9}-\eqref{5.13}, say $(\v,k) \in \mathrm{H}^{1}(0,
		T;\mathbb{H}_{0}^{1}(\mathcal{O}) \cap \mathbb{H}^{2}(\mathcal{O})) \times \mathrm{L}^{2}(0,T)$. To obtain the same, let us set
		$$\mathfrak{T}:=\bigg\{\tau\in(0,T] : (\v,k)\in \mathrm{H}^{1}(0,\tau;\mathbb{H}_{0}^{1}(\mathcal{O}) \cap \mathbb{H}^{2}(\mathcal{O})) \times \mathrm{L}^{2}(0,\tau) \ \mbox{solves the system} \mbox{\eqref{5.9}-\eqref{5.13}} \bigg\}.$$
		Obviously $\mathfrak{T} \neq \emptyset$ from Theorem \ref{Thm1}.
		Owing to Theorem \ref{Thm2}, for all $\tau\in \mathfrak{T}$, the solution is unique and will be denoted by $(\v_{\tau},k_{\tau})$.
		Let  $\tau_{1}$ and $\tau_{2}$ be any two elements of $\mathfrak{T}$ such that $\tau_{1} \leq\tau_{2}$. Then, we have $\v_{{\tau}_{2}}$ and  $k_{{\tau}_{2}}$ is an extension of $\v_{{\tau}_{1}}$ and  $k_{{\tau}_{1}}$, respectively.
		Set $T':= \sup(\mathfrak{T}).$ Of course, $0<T'\leq T$. We need to prove $T'=T$ for the global existence. Let us define
		\begin{equation*}
		\begin{aligned}
		&\v \colon [0,T') \to \mathbb{H}_{0}^{1}(\mathcal{O})\cap \mathbb{H}^{2}(\mathcal{O}), \
		\v(t)=\v_{\tau}(t), \  \mbox{ if } \ t \leq \tau,
		\end{aligned}
		\end{equation*}
		and
		\begin{equation*}
		\begin{aligned}
		&k \colon [0,T') \to \mathbb{R},\ k(t)=k_\tau (t), \ \text{ a.e. } \mbox{ if } \ t \leq \tau.
		\end{aligned}
		\end{equation*}
		Then, $(\v,k)$ is the unique solution of the system \eqref{5.9}-\eqref{5.13} in $[0,\tau]$, $\text{for all} \ \tau \in (0,T')$. Let $T'' := T'/2$.
		Then, by Lemma \ref{Lemma 8c}, there exists $C>0$, such that, $  \text{for all} \ \tau \in [T'',T')$,
		\begin{align}\label{8.18}
		&\|\v\|_{\mathrm{H}^{1}(0,\tau;\mathbb{H}_{0}^{1}(\mathcal{O}) \cap \mathbb{H}^{2}(\mathcal{O}))}\nonumber+\|k\|_{\mathrm{L}^{2}(0,\tau)}\\
		&=\|\v\|_{\mathrm{L}^{2}(0,\tau;\mathbb{H}_{0}^{1}(\mathcal{O}) \cap \mathbb{H}^{2}(\mathcal{O}))}+\|\partial_{t}\v\|_{\mathrm{L}^{2}(0,\tau;\mathbb{H}_{0}^{1}(\mathcal{O}) \cap \mathbb{H}^{2}(\mathcal{O}))}+\|k\|_{\mathrm{L}^{2}(0,\tau)}
		\leq C.
		\end{align}
		As $\tau \to T'$ in \eqref{8.18}, we obtain
		\begin{align*}
		\|\v\|_{\mathrm{L}^{2}(0,T';\mathbb{H}_{0}^{1}(\mathcal{O}) \cap \mathbb{H}^{2}(\mathcal{O}))}+\|\partial_{t}\v\|_{\mathrm{L}^{2}(0,T';\mathbb{H}_{0}^{1}(\mathcal{O}) \cap \mathbb{H}^{2}(\mathcal{O}))}+\|k\|_{\mathrm{L}^{2}(0,T')}
		\leq C.
		\end{align*}
		Therefore, we conclude that $(\v,k)\in \mathrm{H}^{1}(0,T';\mathbb{H}_{0}^{1}(\mathcal{O}) \cap \mathbb{H}^{2}(\mathcal{O})) \times \mathrm{L}^{2}(0,T')$ and solves the system \eqref{5.9}-\eqref{5.13} in $[0,T']$. But this implies that $T'=T$.	If not, by Lemma \ref{Lemma 8a}, we could extend $(\v,k)$ to a solution with domain $[0,T'+\delta]$, but this is not compatible  with the definition of $T'$,	which completes the proof.
	\end{proof}
	
	\medskip\noindent
	{\bf Acknowledgments:} P. Kumar and M. T. Mohan would  like to thank the Department of Science and Technology (DST), India for Innovation in Science Pursuit for Inspired Research (INSPIRE) Faculty Award (IFA17-MA110).

\end{document}